\documentclass[11pt,a4paper]{article}

\usepackage[math]{cellspace}
\usepackage{psfrag}
\usepackage{latexsym}
\usepackage{lscape}
\usepackage{afterpage}
\usepackage[color=green!40]{todonotes}
\usepackage{geometry,amsmath,amssymb,amsthm,graphicx,epsfig,float,cite,rotating,subcaption,comment,bigints,hyperref,appendix,multirow}
\usepackage{mathtools}
\usepackage[normalem]{ulem}

\hypersetup{
    colorlinks=true,
    linkcolor=blue,
    filecolor=magenta,      
    urlcolor=cyan,
    citecolor=red,
}
\DeclareMathOperator*{\argmin}{argmin}
\DeclareMathOperator*{\diag}{diag}
\newcommand{\prox}{\ensuremath{\operatorname{Prox}}}
\newcommand{\Id}{\ensuremath{\operatorname{Id}}}

\newcommand{\ds}{\displaystyle}
\newcommand{\nexto}{\kern -0.54em}
\newcommand{\dR}{\mathbb{R}}

\newcommand{\dZ}{{\cal Z \kern -0.7em Z}}
\newcommand{\dC}{{\rm\hbox{C \kern-0.8em\raise0.2ex\hbox{\vrule
height5.4pt width0.7pt}}}}
\newcommand{\dQ}{{\rm\hbox{Q \kern-0.85em\raise0.25ex\hbox{\vrule
height5.4pt width0.7pt}}}}
\newcommand{\proofbox}{\hspace{\fill}{$\Box$}}

\newtheorem{theorem}{Theorem}

\newtheorem{conjecture}{Conjecture}

\newtheorem{remark}{Remark}

\theoremstyle{definition}
\newtheorem{algorithm}{Algorithm}

\begin{document}

\title{\vspace{-10mm}\bf Douglas--Rachford Algorithm for Control- and State-constrained Optimal Control Problems}

\author{
Regina S. Burachik\thanks{Mathematics, UniSA STEM, University of South Australia, Mawson Lakes, S.A. 5095, Australia. Emails:~regina.burachik@unisa.edu.au, bethany.caldwell@mymail.unisa.edu.au, yalcin.kaya@unisa.edu.au.}
\qquad
Bethany I. Caldwell\footnotemark[1]
\qquad
C. Yal{\c c}{\i}n Kaya\footnotemark[1]
}

\maketitle

\begin{abstract} {\noindent\sf  
We consider the application of the Douglas--Rachford (DR) algorithm to solve linear--quadratic (LQ) control problems with box constraints on the state and control variables.  We split the constraints of the optimal control problem into two sets: one involving the ODE with boundary conditions, which is affine, and the other a box.  We rewrite the LQ control problems as the minimization of the sum of two convex functions.  We find the proximal mappings of these functions which we then employ for the projections in the DR iterations.  We propose a numerical algorithm for computing the projection onto the affine set.  We present a conjecture for finding the costates and the state constraint multipliers of the optimal control problem, which can in turn be used in verifying the optimality conditions.  We carry out numerical experiments with two constrained optimal control problems to illustrate the working and the efficiency of the DR algorithm compared to the traditional approach of direct discretization.}
\end{abstract}
\begin{verse}
\end{verse}


%\centerline{\bf Submitted to }
\pagestyle{myheadings}
%\markboth{}{Submitted to {\sf }, }
\markboth{}{\sf\scriptsize DR Algorithm for Constrained Optimal Control Problems\ \ by R.~S.~Burachik, B.~I.~Caldwell, and C.~Y.~Kaya}

\section{Introduction}\label{sec:intro}
Many of the problems we encounter in the world that can be presented as optimal control problems contain constraints on both the state and control variables. Imagine a scheduling problem where the aim is to maximize profits such as in \cite{MauKimVos2005} in which the control variable is the production rate and the state variable is the inventory level. Naively, one may want to greatly increase both the inventory level and production rate to maximize profits but in practice there is a limit to the amount of inventory a factory can hold and a limit on how quickly humans or machines can work. In order to accurately model this problem, and many others from a wide rage of application areas such as manufacturing, engineering, science, economics, etc., we must use an optimal control problem with both state and control constraints. Although these problems are commonly found in applications they are much more difficult to solve than optimal control problems which purely have control constraints.

We focus our attention on linear--quadratic (LQ) state- and control-constrained optimal control problems. These are infinite dimensional optimization problems that involve the minimization of a quadratic objective function subject to linear differential equation (DE) constraints, affine constraints on the state variables, and affine constraints on the control variables. There is extensive literature studying LQ control problems as they model many problems from a wide variety of disciplines -- see \cite{BurKayMaj2014,MTM,AmmKen1998,MauObe2003,Mou2011,BusMau2000,KugPes1990}. Though many of the LQ control problems posed in the literature contain control constraints along with the linear DE constraints, it is rarer to see state constraints included since, as mentioned above, they are much more difficult to deal with. This paper proposes a unique approach for solving state-constrained problems by applying the Douglas--Rachford (DR) algorithm.

The DR algorithm is a projection algorithm used to minimize the sum of two convex functions. In order to apply the algorithm one needs the proximal operators of the two convex functions. Splitting and projection methods such as the DR algorithm are a popular area of research in optimization with a variety of applications -- see \cite{GraEls2008,AraBorTam2014,BauKoch,AraCampEls2020,Bau2008} for their use in sphere packing problems, protein reconstruction, etc. The use of these methods to solve discrete-time optimal control problems is not new but there are very few applications of these methods to continuous-time optimal control problems. The paper \cite{BausBuraKaya2019} is one in which projection methods were used to solve the energy-minimizing double integrator problem with control constraints. The papers \cite{BurCalKay2022,BurCalKay2024, BurCalKayMou2023} address more general energy-minimizing LQ control problems with control constraints. While a collection of general projection methods is used in \cite{BurCalKay2022}, the DR algorithm is used in~\cite{BurCalKay2024, BurCalKayMou2023}. 

% \sout{Then the authors of the current paper used}  and  \sout{to solve more general energy-minimizing LQ {\color{cyan} control} problems with control constraints}. 

Following the promising numerical results observed in \cite{BurCalKay2024, BurCalKayMou2023}, the present paper will use the DR algorithm for addressing the more challenging control- and state-constrained LQ control problems.

The current technique for solving these control- and state-constrained LQ problems is a direct discretization approach where the infinite-dimensional problem is reduced to a large scale finite-dimensional problem by use of a discretization scheme, e.g. a Runge--Kutta method \cite{Kaya2010,BanKay2013}. This discretized problem is then solved through the use of an optimization modelling language such as AMPL~\cite{AMPL} paired with a (finite-dimensional) large-scale optimization software such as Ipopt~\cite{WacBie2006}. 

In the present paper (as was done in \cite{BausBuraKaya2019, BurCalKay2022} for control constraints only) we make the following contributions for the numerical solution of LQ control problems with state and control constraints.
\begin{itemize}
\item[(i)] We rewrite the original LQ control problem as the minimization of the sum of two convex functions. 
\item[(ii)] We find the proximal mappings of these functions by solving the resulting calculus of variations subproblems (Theorems~\ref{thm:proxf}--\ref{thm:proxg}).
\item[(iii)] We define and apply the DR algorithm to the constrained LQ control problem (Algorithm~\ref{alg:DR}).
\item[(iv)] We propose an algorithm for computing the projection onto the affine constraint set defined by the ODEs (Algorithm~\ref{alg:projA}).
\item[(v)] We carry out numerical experiments to illustrate the efficiency of the DR algorithm over the traditional approach of direct discretization.
\end{itemize}
In addition to these, we also make the following contribution, which is new even for the purely control-constrained LQ control problems.
\begin{itemize}
\item[(vi)] We present a conjecture for finding the costates.  Moreover, we present a technique for finding the multiplier for the state constraints in the case only one state constraint is active at a given time.  We effectively employ the costates and the state control multiplier for a numerical verification of the optimality conditions.
\end{itemize}

The two convex functions mentioned above are defined as follows: the information from the ODE constraints appears in one function while the information from the control and state constraints along with the objective functional appear in the other. We derive the proximal mappings of the two convex functions without reducing the original infinite-dimensional optimization problem to a finite-dimensional one, though we need to discretize the state and control variables over a partition of time when implementing the DR algorithm since a digital computer cannot iterate with functions.

The paper is structured as follows. In Section \ref{sec:oc} we give the problem formulation and optimality conditions for the optimal control problem. In Section \ref{sec:prox} we derive the proximal mappings used in the implementation of the DR algorithm. Section \ref{sec:DR} introduces the DR algorithm. Then Section \ref{sec:experiments} begins with an algorithm for computing one of the proximal mappings (namely the projection onto an affine set) and a conjecture used to obtain the costate variable of the original LQ control problem. Next, this section introduces two example problems, a harmonic oscillator and a mass--spring system. At the end of this section, numerical experiments for the DR algorithm and AMPL--Ipopt suite and their comparisons are given for these two problems.  Finally, concluding remarks and comments for future work are provided in Section \ref{sec:concl}.

\section{Optimal Control Problem}
\label{sec:oc}

In this section we formulate the general optimal control problem that will be the focus of this paper. We give some necessary definitions and provide conditions for optimality from optimal control theory.

Before introducing the optimal control problem we will give some standard definitions. Unless otherwise stated all vectors are column vectors. Let ${\cal L}^2([t_0,t_f];\mathbb{R}^q)$ be the Banach space of Lebesgue measurable functions $z:[t_0,t_f]\rightarrow\mathbb{R}^q$, with finite ${\cal L}^2$ norm, namely,
\[{\cal L}^2([t_0,t_f];\mathbb{R}^q):=\left\{z:[t_0,t_f]\rightarrow\mathbb{R}^q\,\:|\,\:\|z\|_{{\cal L}^2} := \left(\int_{t_0}^{t_f}\|z(t)\|^2\,dt\right)^{1/2}<\infty\right\}\]
where $\|\cdot\|$ is the $\ell_2$ norm in $\dR^q$. Furthermore, ${\cal W}^{1,2}([t_0,t_f];\mathbb{R}^q)$ is the Sobolev space of absolutely continuous functions, namely
\[{\cal W}^{1,2}([t_0,t_f];\mathbb{R}^q):=\left\{z\in {\cal L}^2([t_0,t_f];\mathbb{R}^q)\,|\,\dot{z}:=dz/dt\in {\cal L}^2([t_0,t_f];\mathbb{R}^q)\right\},\]
endowed with the norm
\[\|z\|_{{\cal W}^{1,2}}:=(\|z\|_{{\cal L}^2}^2+\|\dot{z}\|_{{\cal L}^2}^2)^{1/2}.\]
With these definitions we can state now the general LQ control problem as follows.
\begin{align*}
\mbox{(P) }\left\{\begin{array}{rl}
\ds\min & \ \ \ds\frac{1}{2}\int_{t_0}^{t_f} \left(x(t)^TQ(t)x(t) + u(t)^TR(t)u(t)\right)\,dt  \\[5mm] 
\mbox{subject to} & \ \ \dot{x}(t) = A(t)x(t) + B(t)u(t)\,,\ \ x(t_0) = x_0\,,\ \ x(t_f) = x_f\,, \\[2mm]
& \ \ \underline{u}\leq u(t)\leq \overline{u}\,, \ \ \underline{x}\leq x(t)\leq \overline{x}\,.
\end{array} \right.
\end{align*}
The {\em state variable} $x\in {\cal W}^{1,2}([t_0,t_f];\mathbb{R}^n)$, with $x(t) := (x_1(t),\ldots,x_n(t))\in\dR^n$, and the {\em control variable} $u\in {\cal L}^2([t_0,t_f];\mathbb{R}^m)$, with $u(t) := (u_1(t),\ldots,u_m(t))\in\dR^m$. %and $U$ is a fixed closed subset of $\dR^m$ 
The time varying matrices $A:[t_0,t_f]\rightarrow\mathbb{R}^{n\times n}$, $B:[t_0,t_f]\to\mathbb{R}^{n\times m}$, $Q:[t_0,t_f]\rightarrow\mathbb{R}^{n\times n}$ and $R:[t_0,t_f]\rightarrow\mathbb{R}^{m\times m}$ are continuous. For every $t\in [t_0,t_f]$, the matrices $Q(t)$ and $R(t)$ are symmetric, and respectively positive semi-definite and positive definite. For clarity of arguments, these matrices are assumed to be diagonal, namely $Q(t) := \diag(q_1(t),\ldots,q_n(t))$ and $R(t) := \diag(r_1(t),\ldots,r_m(t))$. These diagonality assumptions particularly simplify proximal mapping expressions that appear later. The initial and terminal states are given as $x_0$ and $x_f$ respectively.

\subsection{Optimality conditions}
In this section we state the Maximum Principle by using the {\em direct adjoining approach} from \cite{HarSetVic1995}. We start by defining the extended {\em Hamiltonian function} $H:\dR^n \times \dR^m \times \dR^n \times \dR^n \times \dR^n \times [t_0,t_f]\to \dR$ for Problem~(P) as
\begin{multline*}
H(x(t),u(t),\lambda(t),\mu^1(t),\mu^2(t),t) := \frac{1}{2}(x(t)^TQ(t)x(t)+u(t)^TR(t)u(t) )+\lambda(t)^T(A(t)x(t)\\+B(t)u(t))+\mu^1(t)^T(x(t)-\overline{x}(t))+\mu^2(t)^T(\underline{x}(t)-x(t)),
\end{multline*}
where the {\em adjoint variable} vector $\lambda:[t_0,t_f]\rightarrow\mathbb{R}^n$ with $\lambda(t):=(\lambda_1(t),\dots,\lambda_n(t))\in\mathbb{R}^n$ and the {\em state constraint multiplier vectors} $\mu^1,\mu^2:[t_0,t_f]\rightarrow\mathbb{R}^n$ with $\mu^i(t):=(\mu_1^i(t),\dots,\mu_n^i(t))\in\mathbb{R}^n$, $i=1,2$. For brevity, we use the following notation,
\[
H[t] := H(x(t),u(t),\lambda(t),\mu^1(t),\mu^2(t),t)\,.
\]
The adjoint variable vector is assumed to satisfy
\begin{equation}\label{eqn:adj}
    \dot{\lambda}(t) := -H_x[t] = -Q(t)x(t)-A(t)^T\lambda(t)-\mu^1(t)+\mu^2(t)
\end{equation}
for a.e. $t\in[t_0,t_f]$, where $H_x:=\partial H/\partial x$. 

\noindent
{\bf Maximum Principle.}\ \ Suppose the pair 
\[
(x,u)\in{\cal W}^{1,2}([t_0,t_f];\mathbb{R}^n)\times{\cal L}^2([t_0,t_f];\mathbb{R}^m)
\]
is optimal for Problem (P). Then there exists a piecewise continuous adjoint variable vector $\lambda\in{\cal W}^{1,2}([t_0,t_f];\mathbb{R}^n)$ as defined in \eqref{eqn:adj} and piecewise continuous multipliers $\mu^1,\mu^2\in{\cal L}^2([t_0,t_f];\mathbb{R}^n)$, such that $(\lambda(t),\mu^1(t),\mu^2(t))\neq\mathbf{0}$ for all $t\in[t_0,t_f]$, and that for a.e. $t\in[t_0,t_f]$,
\begin{align}
    u_i(t) &= \argmin_{\underline{u}_i\leq v_i\leq \overline{u}_i} H(x(t),u_1(t),\ldots,v_i,\dots,u_m(t),\lambda(t),\mu^1(t),\mu^2(t),t) \nonumber \\
    &= \argmin_{\underline{u}_i\leq v_i\leq \overline{u}_i} \frac{1}{2} r_i(t)v_i^2 + \lambda(t)^Tb_i(t)v_i\,,
    \label{eqn:maxPrin}
\end{align}
for $i = 1,\ldots,m$, where $b_i(t)$ is the $i$th column of the matrix $B(t)$ and $r_i(t)$ is the $i$th diagonal element of $R(t)$. 
Moreover, the multipliers $\mu^1(t),\mu^2(t)$ must satisfy the {\em complementarity conditions}
\begin{align}
 \mu_i^1(t)\geq 0, \qquad \mu_i^1(t)(x_i(t)-\overline{x}_i(t))=0 \label{eqn:mu1}\\[2mm]
 \mu_i^2(t)\geq 0, \qquad \mu_i^2(t)(\underline{x}_i(t)-x_i(t))=0 \label{eqn:mu2}
\end{align}
for all $i=1,\dots,n$.

Suppose $\underline{u}_i=-\infty$, $\overline{u}_i=\infty$, $i = 1,\ldots,m$, i.e., the control vector is unconstrained. Then \eqref{eqn:maxPrin} becomes
\[H_{u_i}[t] = 0,\]
i.e.,
\begin{equation}  \label{ui_unconstr}
r_i(t)u_i(t)+b_i(t)^T\lambda(t) = 0\,,
\end{equation}
$i = 1,\ldots,m$. Then \eqref{ui_unconstr} can be solved for $u_i(t)$ as
\begin{equation}  
u_i(t) = -\frac{1}{r_i(t)}b_i(t)^T\lambda(t)\,,
\end{equation}
for $i = 1,\ldots,m$; or using the matrices $B(t),R(t)$,
\begin{equation} \label{eqn:u(t)}
u(t) = -[R(t)]^{-1}B(t)^T\lambda(t)\,.
\end{equation}
When we consider the constraints on $u$, one gets from \eqref{eqn:maxPrin},
\begin{eqnarray}\label{eqn:u_gen}
u_j(t) = \left\{\begin{array}{rl}
   \underline{u}_j, & \mbox{if } \frac{1}{r_j(t)}b_j(t)^T\lambda(t) \geq \underline{u}_j, \\[2mm]
  -\frac{1}{r_j(t)}b_j(t)^T\lambda(t), & \mbox{if } \underline{u}_j\leq -\frac{1}{r_j(t)}b_j(t)^T\lambda(t) \leq \overline{u}_j, \\[2mm]
  \overline{u}_j, & \mbox{if } \frac{1}{r_j(t)}b_j(t)^T\lambda(t) \leq \overline{u}_j,
\end{array} \right.
\end{eqnarray}
for a.e. $t\in[t_0,t_f]$, $j=1,\dots,m$. \\ \\
% \begin{align}
%  % \mu_i^1(t)\geq 0, \qquad \mu_i^1(t)(x_i(t)-\overline{x}_i(t))=0 \label{eqn:mu1}\\[2mm]
%  % \mu_i^2(t)\geq 0, \qquad \mu_i^2(t)(\underline{x}_i(t)-x_i(t))=0 \label{eqn:mu2}
% \end{align}

\section{Splitting and Proximal Mappings}
\label{sec:prox}

In this section, we rewrite Problem (P) as the minimization of the sum of two convex functions $f$ and $g$, and give the proximal mappings for these functions in Theorems \ref{thm:proxf}--\ref{thm:proxg}.

We split the constraints from (P) into two sets $\cal{A},\cal{B}$ given as
\begin{eqnarray} 
{\cal A} := &&\big\{(x,u)\in {\cal L}^{2}([t_0,t_f];\dR^n)\times {\cal L}^2([t_0,t_f];\dR^m)\ | \mbox{ which solves } \nonumber \\
&&\dot{x}(t) = A(t)x(t)+B(t)u(t), \ \ x(t_0) = x_0, \ \ x(t_f) = x_f, \mbox{ for all } t\in[t_0,t_f]\big\}\,, \label{A_gen} \\[2mm]
{\cal B} := &&\big\{(x,u)\in {\cal L}^{2}([t_0,t_f];\dR^n)\times {\cal L}^2([t_0,t_f];\dR^m)\ |\ \underline{x}_i\le x_i(t)\le \overline{x}_i\,,\ \underline{u}_j\le u_j(t)\le \overline{u}_j\,, \nonumber\\[1mm]
&&\mbox{for all } t\in[t_0,t_f],\ i=1,\dots,n,\ 
j=1,\dots,m\big\}\,.\label{B_gen}
\end{eqnarray}
Despite previously defining $x\in{\cal W}^{1,2}([t_0,t_f];\mathbb{R}^n)$ in our sets ${\cal A}, {\cal B}$ we let $x\in{\cal L}^2([t_0,t_f];\mathbb{R}^n)$ to simplify the calculation of the proximal mappings. By definition of ${\cal W}^{1,2}([t_0,t_f];\mathbb{R}^n)$ if $x\in{\cal A}$ then $x\in{\cal W}^{1,2}([t_0,t_f];\mathbb{R}^n)$.

We assume that the control system $\dot{x}(t) = A(t)x(t) + B(t)u(t)$ is {\em controllable}; in other words the control system can be driven from any $x_0$ to any other $x_f$---for a precise definition of controllability and the tests for controllability, see \cite{Rugh1996}. Then there exists a (possibly not unique) $u(\cdot)$ such that, when this $u(\cdot)$ is substituted, the boundary-value problem given in the expression for ${\cal A}$ has a solution $x(\cdot)$.  In other words, ${\cal A} \neq \emptyset$.  Also, clearly, ${\cal B} \neq \emptyset$.  We note that the constraint set~${\cal A}$ is an {\em affine subspace}. Given that ${\cal B}$ is a {\em box}, the constraints turn out to be two convex sets in Hilbert space. Since every sequence converging in ${\cal L}^2$ has a subsequence converging pointwise, it is straightforward to see that the set ${\cal B}$ is closed in ${\cal L}^2$. The closedness of ${\cal A}$ will be established later on as a consequence of Theorem \ref{thm:proxg} (see Remark~\ref{rem:A closed}).

Fix $\beta>0$ and let
\begin{equation}  \label{def:fg}
f(x,u) := \iota_{\cal{B}}(x,u) + \frac{\beta}{2}(x(t)^TQ(t)x(t) + u(t)^TR(t)u(t))\quad\mbox{and}\quad g(x,u) := \iota_{\cal{A}}(x,u)\,,
\end{equation}
where $\iota_{\cal{C}}$ is the indicator function of the set $\cal{C}$, namely
\[
\iota_{\cal C}(x,u) := \left\{\begin{array}{ll}
    0\,, & \mbox{ if\ \ } (x,u)\in{\cal C}\,, \\
    \infty\,, & \mbox{ otherwise}.
    \end{array}\right.
\]

Problem (P) is then equivalent to the following problem.
\begin{equation}  \label{minfg}  
\min \ \ f(x,u) + g(x,u).
\end{equation}

In our setting, we assume that we are able to compute the projector operators $P_{\cal{A}}$ and $P_{\cal{B}}$.  These operators project a given point onto each of the constraint sets $\cal{A}$ and $\cal{B}$, respectively. Recall that the {\em proximal mapping} of a functional $h$ is defined by \cite[Definition~24.1]{BauCom2017}.  For our setting,
\begin{equation}  \label{def:prox}
\prox_h(x,u) := \argmin_{y\in {\cal L}^2([t_0,t_f];\dR^n) \atop v\in {\cal L}^2([t_0,t_f];\dR^m)}
\left(h(y,v) + \frac{1}{2}\|y - x\|_{{\cal L}^2}^2 + \frac{1}{2}\|v - u\|_{{\cal L}^2}^2  \right),
\end{equation}
for any $(x,u)\in {\cal L}^2([t_0,t_f];\dR^n)\times {\cal L}^2([t_0,t_f];\dR^m)$. Recall that the {\em projection $P_{\cal C}(x,u)$ of a point $(x,u)$ onto ${\cal C}$} is characterized by $P_{\cal C}(x,u)\in {\cal C}$ and, $\forall (y,v)\in {\cal C}$, $\langle (y,v)-P_{\cal C}(x,u)|(x,u)-P_{\cal C}(x,u)\rangle\leq0$ \cite[Theorem~3.16]{BauCom2017}.

Note that $\prox_{\iota_{\cal C}}=P_{\cal C}$.

In order to implement the Douglas--Rachford algorithm we must write the proximal mappings $f$ and $g$. The proofs of Theorems~\ref{thm:proxf} and~\ref{thm:proxg} below follow the broad lines of proof of Lemma~2 in~\cite{BurCalKayMou2023}.  In both theorems, the major difference from~\cite{BurCalKayMou2023} is that the proximal operators in the current paper have two variables $x^-$ and $u^-$.  Thanks to separability, the proof of Theorem~\ref{thm:proxf} is a straightforward modification of the corresponding part of the proof of~\cite[Lemma~2]{BurCalKayMou2023}. We include a full proof of Theorem~\ref{thm:proxf} for the convenience of the reader.  On the other hand, the proof of Theorem~\ref{thm:proxg} deals with the solution of a more involved optimal control subproblem, namely Problem~(Pg).

\begin{theorem}\label{thm:proxf}
The proximal mapping of $f$
%$$f(x,u) = \iota_{\cal{B}}(x,u) + \frac{\beta}{2}\int_{t_0}^{t_f}(x(t)^TQ(t)x(t) + u(t)^TR(t)u(t))dt$$ 
is given as $\prox_f(x^-,u^-) = (y,v)$ such that the components of $y$ and $v$ are expressed as
\begin{eqnarray}\label{eqn:ProjB_x}
y_i(t) = \left\{\begin{array}{rl}
   \overline{x}_i, & \mbox{if } \frac{1}{\beta q_i + 1}x^-_i(t) \geq \overline{x}_i, \\[2mm]
  \frac{1}{\beta q_i + 1}x^-_i(t), & \mbox{if } \underline{x}_i\leq \frac{1}{\beta q_i + 1}x^-_i(t) \leq \overline{x}_i, \\[2mm]
  \underline{x}_i, & \mbox{if } \frac{1}{\beta q_i + 1}x^-_i(t) \leq \underline{x}_i,
\end{array} \right.
\end{eqnarray}
\begin{eqnarray}\label{eqn:ProjB_u}
v_j(t) = \left\{\begin{array}{rl}
   \overline{u}_j, & \mbox{if } \frac{1}{\beta r_j + 1}u^-_j(t) \geq \overline{u}_j, \\[2mm]
  \frac{1}{\beta r_j + 1}u^-_j(t), & \mbox{if } \underline{u}_j\leq \frac{1}{\beta r_j + 1}u^-_j(t) \leq \overline{u}_j, \\[2mm]
  \underline{u}_j, & \mbox{if } \frac{1}{\beta r_j + 1}u^-_j(t) \leq \underline{u}_j,
\end{array} \right.
\end{eqnarray}
for all $t\in[t_0,t_f],\ i=1,\dots,n, \ j=1,\dots,m$.
\end{theorem}
\begin{proof}
From 
%\cite[Eqn. (24.1)]{BauCom2017} 
\eqref{def:prox} and the definition of $f$ in~\eqref{def:fg} we have that
\begin{align*}
\prox_f(x^-,u^-) = &\argmin_{x,u} \iota_{\cal{B}}(x,u) + \frac{1}{2}\int_{t_0}^{t_f} \Bigl(\beta x(t)^TQ(t)x(t)+\|x(t)-x^-(t)\|^2 \\[2mm] &+ \beta u(t)^TR(t)u(t) + \|u(t)-u^-(t)\|^2\Bigr)\,dt.
\end{align*}
In other words, to find $\prox_f(x^-,u^-)$ we need to find $(y,v)$ that solves
\begin{align*}
\mbox{(Pf) }\left\{\begin{array}{rl}
\ds\min & \ \ \ds\frac{1}{2}\int_{t_0}^{t_f} \Bigl(\beta y(t)^TQ(t)y(t)+\|y(t)-x^-(t)\|^2 \\[4mm] & \hspace{10mm} + \beta v(t)^TR(t)v(t) + \|v(t)-u^-(t)\|^2\Bigr)\,dt  \\[5mm] 
\mbox{subject to} & \ \ \underline{x}_i\le y_i(t)\le \overline{x}_i\,,\ \underline{u}_j\le v_j(t)\le \overline{u}_j \\[2mm]& \ \ \mbox{for all $t\in[t_0,t_f],\ i=1,\dots,n, \ j=1,\dots,m$.}
\end{array} \right.
\end{align*}
Problem~(Pf) is separable in the variables $y$ and $v$ so we can consider the problems of minimizing w.r.t. $y$ and $v$ individually and thus solve the two subproblems
\begin{align*}
\mbox{(Pf1) }\left\{\begin{array}{rl}
\ds\min & \ \ \ds\frac{1}{2}\int_{t_0}^{t_f} \beta y(t)^TQ(t)y(t)+\|y(t)-x^-(t)\|^2\,dt  \\[5mm] 
\mbox{subject to} & \ \ \underline{x}_i\le y_i(t)\le \overline{x}_i\,, \ \ \mbox{for all $t\in[t_0,t_f],\ i=1,\dots,n$,}
\end{array} \right.
\end{align*}
and
\begin{align*}
\mbox{(Pf2) }\left\{\begin{array}{rl}
\ds\min & \ \ \ds\frac{1}{2}\int_{t_0}^{t_f} \beta v(t)^TR(t)v(t) + \|v(t)-u^-(t)\|^2\,dt  \\[5mm] 
\mbox{subject to} & \ \ \underline{u}_j\le v_j(t)\le \overline{u}_j \ \ \mbox{for all $t\in[t_0,t_f],\ j=1,\dots,m$.}
\end{array} \right.
\end{align*}
The solution to Problem~(Pf1) is given by
\[
y_i(t) = \argmin_{\underline{x}_i \le z_i \le \overline{x}_i} \big(\, \beta\,q_i\,z_i^2 + (z_i - x_i^-(t))^2\,\big)\,,
\]
$i = 1,\ldots,n$, which, after straightforward manipulations, yields~\eqref{eqn:ProjB_x}.  The solution to Problem~(Pf2) is obtained similarly as
\[
v_j(t) = \argmin_{\underline{u}_j \le w_j \le \overline{u}_j} \big(\, \beta\,r_j\,w_j^2 + (w_j - u_j^-(t))^2\,\big)\,,
\]
$j = 1,\ldots,m$, which yields~\eqref{eqn:ProjB_u} after straightforward manipulations.
\end{proof}

\begin{theorem}\label{thm:proxg}
The proximal mapping of $g$ %$g(x,u) = \iota_{\cal{A}}(x,u)$ 
is given as $\prox_g(x^-,u^-) = P_{\cal{A}}(x^-,u^-) = (y,v)$ such that
\begin{align}
    y(t) =&\ x(t), \label{eqn:ProjA_x} \\
    v(t) =&\ u^-(t)-B(t)^T\lambda(t), \label{eqn:ProjA_u}
\end{align}
where $x(t), \lambda(t)$ are obtained by solving the two-point boundary-value problem (TPBVP)
\begin{equation}  \label{TPBVP}
\begin{array}{ll}
&\dot{x}(t) = A(t)x(t) + B(t)u^-(t) - B(t)B(t)^T\lambda(t)\,, \ \ x(t_0) = x_0\,,\ \ x(t_f) = x_f\,, \\[2mm]
&\dot{\lambda}(t) = -x(t) + x^-(t) - A(t)^T\lambda(t)\,.
\end{array}
\end{equation}
\end{theorem}
\begin{proof}
Using \cite[Example 12.25]{BauCom2017}, and the definition of $g$ in~\eqref{def:fg},
\[\prox_g(x^-,u^-) = \prox_{\iota_{\cal{A}}}(x^-,u^-) = P_{\cal{A}}(x^-,u^-),\]
which verifies the very first assertion.
From 
%\cite[Eqn. (24.1)]{BauCom2017} 
\eqref{def:prox} and the definition of $g$ in~\eqref{def:fg} we have that
\[\prox_g(x^-,u^-) = \argmin_{x,u} \iota_{\cal{A}}(x,u) + \frac{1}{2}\int_{t_0}^{t_f}(\|x(t)-x^-(t)\|^2 + \|u(t)-u^-(t)\|^2)\,dt.\]
In other words, to find $\prox_g(x^-,u^-)$ we need to find $(y,v)$ that solves the problem
\begin{align*}
\mbox{(Pg) }\left\{\begin{array}{rl}
\ds\min & \ \ \ds\frac{1}{2}\int_{t_0}^{t_f} \|y(t)-x^-(t)\|^2 + \|v(t)-u^-(t)\|^2\,dt  \\[5mm] 
\mbox{subject to} & \ \ \dot{y}(t) = A(t)y(t) + B(t)v(t)\,, \ \ y(t_0) = x_0\,,\ \ y(t_f) = x_f\,.
\end{array} \right.
\end{align*}
Problem~(Pg) is an optimal control problem where $y(t)$ is the state variable and $v(t)$ is the control variable.  The Hamiltonian for Problem (Pg) is
\[
H(y(t),v(t),\lambda(t),t) := \frac{1}{2}(\|y(t)-x^-(t)\|^2 + \|v(t)-u^-(t)\|^2) + \lambda(t)^T(A(t)y(t) + B(t)v(t))
\]
and the associated costate equation is
\begin{equation}  \label{Pg_costate}
\dot{\lambda}(t) = -\partial H / \partial y= -y(t)+x^-(t)-A(t)^T\lambda(t).
\end{equation}
If $v$ is the optimal control for Problem~(Pg) then, by the maximum principle, $H_v[t] = 0$ for all $t\in[t_0,t_f]$.  In other words,
\[
v(t) - u^-(t) + B(t)^T\lambda(t) = 0\,,
\]
for all $t\in[t_0,t_f]$, a re-arrangement of which yields~\eqref{eqn:ProjA_u}. Collecting together the ODE in Problem~(Pg) and the ODE in~\eqref{Pg_costate}, substituting $v(t)$ from~\eqref{eqn:ProjA_u}, and assigning $y(t) = x(t)$, result in the TPBVP in~\eqref{TPBVP}.
\end{proof}

\begin{remark} \rm
We note from Theorem~\ref{thm:proxg} that $\prox_g$ is the projection onto the affine set $\cal A$.  Unlike $\prox_f$, in general, we cannot find an analytical solution to \eqref{TPBVP} to obtain $\prox_g$ (or the projection onto $\cal A$) so for this purpose we will propose Algorithm~\ref{alg:projA} in Section~\ref{sec:algorithm_PA} to get a numerical solution.
\proofbox
\end{remark}

\begin{remark} \label{rem:A closed}\rm
From Theorem~\ref{thm:proxg} we see that every pair $(x,u)\in {\cal L}^2\times {\cal L}^2$ has a projection onto ${\cal A}$. In other words, ${\cal A}$ is a {\em Chebyshev set}. It is well-known that every Chebyshev set is closed (see \cite[Remark 3.11(i)]{BauCom2017}). Hence, the set ${\cal A}$ is closed in the topology of ${\cal L}^2\times {\cal L}^2$.
\proofbox
\end{remark}

\section{Douglas--Rachford Algorithm}
\label{sec:DR}

The application of the Douglas--Rachford (DR) algorithm to our problem is slightly different to that in \cite{BausBuraKaya2019,BurCalKay2024}. Since we are solving control- and state-constrained optimal control problems we must define the proximal mappings at the pair $(x,u)$, rather than just at $u$ as in~\cite{BausBuraKaya2019,BurCalKay2024}. Thus in the implementation of the DR algorithm we give the iterations for the pair of iterates $(x^k,u^k)$ rather than for $u^k$ alone.

Given $\beta>0$, we specialize the DR algorithm (see
\cite{DougRach}, \cite{LM} and \cite{EckBer}) 
to the case of minimizing the sum of the two functions $f$ and $g$ as in~\eqref{def:fg}--\eqref{minfg}. The DR operator associated with the ordered pair $(f,g)$ is defined by
\[ 
T :=  \Id -  \prox_f +  \prox_g(2 \prox_f-\Id)\,. 
\]
Application of the operator to our case is given by
\begin{align}\label{eqn:DR}
T(x,u) &= (x,u)-\prox_f(x,u)+P_{\cal{A}}(2\prox_f(x,u)-(x,u))\,,
\end{align}
where the proximal mappings of $f$ and $g$ are provided as in Theorems \ref{thm:proxf}--\ref{thm:proxg}.
Let $X$ be an arbitrary Hilbert space. Now fix $x_0\in X$. Given $x_n\in X$, $k\geq 0$, the DR iterations are set as follows.
\begin{align*}
(b_{x,k},b_{u,k}) &:= \prox_f(x_k,u_k)\,,\\
(x_{k+1},u_{k+1}) &:= T(x_k,u_k)
= (x_k,u_k)-(b_{x,k},b_{u,k})
+P_{\cal{A}}(2(b_{x,k},b_{u,k})-(x_k,u_k))\,.
\end{align*}
The DR algorithm is implemented as follows. We define a new parameter $\gamma:=1/(1+\beta)$ where $\beta$ is the parameter multiplying the objective as in \eqref{def:fg} and Theorem~\ref{thm:proxf}. The choice of $\gamma \in ]0,1[$ can be made because changing $\beta$ does not affect the solution of Problem~(P).

\noindent
\begin{algorithm}{({\bf Douglas--Rachford})}\label{alg:DR}
\begin{description}
\item[Step 1] ({\em Initialization}) Choose a parameter $\gamma\in\left]0,1\right[$ and the initial iterate $(x^0,u^0)$ arbitrarily. 
Choose a small parameter $\varepsilon>0$, and set $k=0$. 
\item[Step 2] ({\em Projection onto ${\cal B}$})  Set $(x^-,u^-) = \gamma\, (x^k,u^{k})$. 
Compute $(\widetilde{x},\widetilde{u}) = P_{{\cal B}}(x^-,u^-)$ by using \eqref{eqn:ProjB_x}--\eqref{eqn:ProjB_u}.
\item[Step 3] ({\em Projection onto ${\cal A}$}) Set $(x^-,u^-) := 2\,(\widetilde{x},\widetilde{u})-(x^k,u^k)$. 
Compute $(\widehat{x},\widehat{u}) = P_{{\cal A}}(x^-,u^-)$ by using \eqref{eqn:ProjA_x}--\eqref{eqn:ProjA_u} or Algorithm~\ref{alg:projA}.
\item[Step 4] ({\em Update}) Set $(x^{k+1},u^{k+1}) := (x^k,u^k) + (\widehat{x},\widehat{u}) - (\widetilde{x},\widetilde{u})$.
\item[Step 5] ({\em Stopping criterion}) If $\max(\|x^{k+1} - x^k\|_{L^\infty}, \|u^{k+1} - u^k\|_{L^\infty}) \le \varepsilon$ or $k>=200$, then return $(\widetilde{x},\widetilde{u})$ and stop.  
Otherwise, set $k := k+1$ and go to Step 2.
\end{description}
\end{algorithm}

\begin{remark}  \rm
We point out that, in general, only weak convergence is guaranteed for the DR algorithm (see \cite[Theorem 1]{Svaiter} or \cite[Theorem 4.4]{BauMoursi}).  It should be noted that the final iterate of the state--control pair in the box $\cal B$ is returned by the algorithm.
\proofbox
\end{remark}

\subsection{Algorithm for projector onto $\cal A$}
\label{sec:algorithm_PA}

We introduce a procedure for numerically projecting onto $\cal A$ that is an extension of Algorithm 2 from \cite{BurCalKay2024} to the case of LQ control problems with state and control constraints.  The procedure below~(Algorithm~\ref{alg:projA}) can be employed in Step~3 of the DR algorithm.  In the procedure we effectively solve the TPBVP in~\eqref{TPBVP} by implementing the standard shooting method~\cite{Ascher95, Keller68, Stoer93}.
Throughout the steps of Algorithm~\ref{alg:projA} below, we will solve the ODEs in \eqref{TPBVP}, rewritten here in matrix form as
\begin{align}\label{eqn:lin_sys}
\begin{bmatrix} \dot{x}(t) \\ \dot{\lambda}^{DR}(t) \end{bmatrix} &= \begin{bmatrix} A(t) & -B(t)B(t)^T \\ -I_{n\times n} & -A(t)^T \end{bmatrix}\begin{bmatrix} x(t) \\ \lambda^{DR}(t) \end{bmatrix} + \begin{bmatrix} 0_{n\times n} & B(t) \\ I_{n\times n} & 0_{n\times m} \end{bmatrix}\begin{bmatrix} x^-(t) \\ u^-(t) \end{bmatrix},
\end{align}
with various initial conditions (IC):
\begin{align}\label{eqn:IC}
\mbox{(i)} \begin{bmatrix}x(t_0) \\ \lambda^{DR}(t_0)\end{bmatrix} = \begin{bmatrix}x_0 %\\ 0_{n\times m} \end{bmatrix}, \ \
\\ 0 \end{bmatrix}, \ \
\mbox{(ii)} \begin{bmatrix}x(t_0) \\ \lambda^{DR}(t_0)\end{bmatrix} = \begin{bmatrix}x_0 \\ e_i \end{bmatrix}, \ \
\mbox{(iii)}\begin{bmatrix}x(t_0) \\ \lambda^{DR}(t_0)\end{bmatrix} = \begin{bmatrix}x_0 \\ \lambda^{DR}_0 \end{bmatrix}.
\end{align}
In the above equations, we are using $\lambda^{DR}$, instead of just $\lambda$, to emphasize the fact that $\lambda^{DR}$ is the costate variable emanating from solving Problem~(Pg) to compute the projection onto $\cal A$ within the DR algorithm.  We reiterate that Problem~(Pg) is more involved than its counterpart in~\cite{BurCalKayMou2023}, which leads to the ODE in~\eqref{eqn:lin_sys} which in turn is more complicated than its counterpart in~\cite{BurCalKay2024}.

\vspace{4mm}

\begin{algorithm}{({\bf Numerical Computation of the Projector onto ${\cal A}$})} \label{alg:projA}\
\begin{description}
\item[Step 0] ({\em Initialization}) The following are given: Current iterate $u^-$, the system and control matrices $A(t)$ and $B(t)$, the numbers of state and control variables $n$ and $m$, and the initial and terminal states $x_0$ and $x_f$, respectively. Define $z(t,\lambda_0):=x(t)$.
\item[Step 1] ({\em Near-miss function}) Solve \eqref{eqn:lin_sys} with ICs in \eqref{eqn:IC}(i) to find $z(t_f,0) = x(t_f)$.  \\ Set $\varphi(0) :=  z(t_f,0)-x_f$.
\item[Step 2] ({\em Jacobian}) For $i = 1,\ldots,n$, solve \eqref{eqn:lin_sys} with ICs in \eqref{eqn:IC}(ii), to get $z(t_f,e_i)$. \\ 
Set $\beta_i(t) := z(t_f,e_i) - z(t_f,0)$ and $J_\varphi(0) := \left[\beta_1(t)\ |\ \dots\ |\ \beta_n(t) \right]$.
\item[Step 3] ({\em Missing IC}) Solve\ \ $J_{\varphi}(0)\,\lambda^{DR}_0 := -\varphi(0)$\ \ for\ \ $\lambda^{DR}_0$. 
\item[Step 4] ({\em Projector onto ${\cal A}$}) Solve \eqref{eqn:lin_sys} with ICs in \eqref{eqn:IC}(iii) to find $x(t)$ and $\lambda^{DR}(t)$.  \\ 
Set  $P_{\cal{A}}(x^-,u^-)(t) := (y(t),v(t))$ where $y(t) = x(t)$ and $v(t) = u^-(t)-B(t)^T\lambda^{DR}(t)$.
\end{description}
\end{algorithm}

\subsection{A conjecture for the costates for Problem~(P)}
Recall that the optimal control for Problem~(P) is given by cases in~\eqref{eqn:u_gen}.
% \begin{eqnarray}\label{eqn:u}
% u_j(t) = \left\{\begin{array}{rl}
%    \underline{u}_j, & \mbox{if } \frac{1}{r_j(t)}b_j(t)^T\lambda(t) \geq \underline{u}_j, \\[2mm]
%   -\frac{1}{r_j(t)}b_j(t)^T\lambda(t), & \mbox{if } \underline{u}_j\leq -\frac{1}{r_j(t)}b_j(t)^T\lambda(t) \leq \overline{u}_j, \\[2mm]
%   \overline{u}_j, & \mbox{if } \frac{1}{r_j(t)}b_j(t)^T\lambda(t) \leq \overline{u}_j,
% \end{array} \right.
% \end{eqnarray}
% for all $t\in[t_0,t_f], \ j=1,\dots,m$. \\[5mm]
A \emph{junction time} $\overline{t}_j$ is a time when the control $u_j(t)$ falls into two cases of \eqref{eqn:u_gen} simultaneously, i.e. a point in time where a control constraint transitions from ``active'' to ``inactive,''  or vice versa.  This definition of a junction time becomes important in the following conjecture, which has been formulated and tested by means of extensive numerical experiments.

\begin{conjecture}\label{con:costate}
Let $\lambda^{DR}(t)$ be the costate variable emerging from the projector into $\cal{A}$ computed in Algorithm \ref{alg:projA} and $\lambda(t)$ be the costate variable emanating from Problem (P).  Let $\overline{t}_j$ be a junction time for some $u_j$, $j=1,\dots,m$, such that $b_j(\overline{t}_j)^T\lambda^{DR}(\overline{t}_j) \neq 0$. Define
\[
\alpha:=\dfrac{-r_j(\overline{t}_j)u_j(\overline{t}_j)}{b_j(\overline{t}_j)^T\lambda^{DR}(\overline{t}_j)}\,.
\]
Then
\begin{equation}\label{eqn:conj}
\lambda(t) = \alpha \lambda^{DR}(t)\,. 
\end{equation}
% where {\color{cyan} the scalar} $\alpha = \dfrac{-r_j(\overline{t}_j)u_j(\overline{t}_j)}{b_j(\overline{t}_j)^T\lambda^{DR}(\overline{t}_j)}$.
\end{conjecture}

\begin{remark} \rm
The ability to obtain the costate variable by Conjecture \ref{con:costate} is desirable as a tool for checking that the necessary condition of optimality in~\eqref{eqn:u_gen} is satisfied. Without this conjecture we are unable to verify whether the optimality condition is satisfied when using the DR algorithm, except when a dual version of the DR algorithm is employed, as in \cite{BurCalKayMou2023}.
\proofbox
\end{remark}

Once we have calculated $\lambda$ in this way we can also find a multiplier $\mu^k$, $k = 1,2$, numerically for the case when only one state constraint is active at a given time. Suppose that only the $i$th state box constraint becomes active. By rearranging Equation \eqref{eqn:adj}, using numerical differentiation to find $\dot{\lambda}$ and assuming $\mu_i^2(t)=0$,
\begin{equation}  \label{eqn:mu1_ex}
\mu_i^1(t) = -Q(t)x(t)-A(t)^T\lambda(t)-\dot{\lambda}(t).
\end{equation}
If $\mu_i^1(t)=0$, then we compute
\begin{equation}\label{eqn:mu2_ex}
\mu_i^2(t) = Q(t)x(t)+A(t)^T\lambda(t)+\dot{\lambda}(t).
\end{equation}
With \eqref{eqn:mu1_ex}, or with \eqref{eqn:mu2_ex}, the complementarity conditions in~\eqref{eqn:mu1} or \eqref{eqn:mu2} can now be checked numerically.

\section{Numerical Experiments}\label{sec:experiments}

We will now introduce two example problems. Along with posing the optimal control problems we also give plots of their optimal controls, states, costates and multipliers with vertical lines signifying the regions where the state constraints become active.

\subsection{Harmonic oscillator}
Problem (PHO) below contains the dynamics of a harmonic oscillator which is typically used to model a point mass with a spring. The dynamics are given as $\ddot{y}(t) + \omega_0^2 y(t) = f(t)$ where $\omega_0 > 0$ is known as the natural frequency and $f(t)$ is some forcing. In a physical system $y$ represents the position of a unit mass, $\dot{y}$ is the velocity of said mass, the natural frequency is expressed as $\omega_0=\sqrt{k}$ where $k$ is the stiffness of the spring producing the harmonic motion and $f$ is the restoring force. In addition to the restoring force we will introduce another force $u_1$ that will affect the velocity $\dot{y}$ directly. We let $x_1:=y$, $x_2:=\dot{y}$ and $u_2:=f$ to arrive at $\dot{x}_1(t) = x_2(t)+u_1(t)$, $\dot{x}_2(t) = -\omega_0^2x_1(t)+u_2(t)$.

In this example problem the objective contains the squared sum of all four variables in the system. It is common to see this problem with the objective of minimizing the energy of the control variable but in this case we have also included the state variables to test the algorithm with a slightly more involved objective. The focus of this research is control- and state-constrained problems so the constraints are added as in Problem~(P).

\begin{align*}
\mbox{(PHO) }\left\{\begin{array}{rl}
\ds\min & \ \ \ds\frac{1}{2}\int_{0}^{2\pi} x_1^2(t) + x_2^2(t) + u_1^2(t) + u_2^2(t)\,dt  \\[5mm] 
\mbox{subject to} & \ \ \dot{x}_1(t) = x_2(t) + u_1(t)\,,\ \ x_1(0) = 0\,,\ \ x_2(2\pi) = 0\,, \\[2mm]
& \ \ \dot{x}_2(t) = -4x_1(t) + u_2(t)\,,\ \ x_2(0) = 1\,,\ \ x(2\pi) = 0\,, \\[2mm]
& \ \ \underline{u}\leq u(t)\leq \overline{u}\,, \ \ \underline{x}\leq x(t)\leq \overline{x}
\end{array} \right.
\end{align*}

\subsection{Simple spring--mass system}

The simple spring-mass system is another physical system that can be easily visualised, see \cite{SchSco2021}. This problem contains two masses and two springs connected in sequence with dynamics given by $m_1\ddot{y}_1(t)+(k_1+k_2)y_1(t)-k_2y_2=f_1(t)$, $m_2\ddot{y}_2(t)-k_2y_1(t)+k_2y_2(t)=f_2(t)$ where $m_1,m_2$ are the two masses, $k_1,k_2$ are the spring coefficients (stiffness) and $f_1(t),f_2(t)$ are the forces applied to $m_1,m_2$ respectively. Let $x_1:=y_1$, $x_2:=\dot{y}_1$, $x_3:=y_2$, $x_4:=\dot{y}_2$, $u_1 := f_1$, $u_2 := f_2$, $m_1=m_2=1$, $k_1=1$ and $k_2=2$ then we retrieve the system in Problem (PSM). This dynamical system furnishes an optimal control problem with four state variables and two control variables. As in (PHO) we add state and control constraints and set the objective as the integral of the squared sum of the state and control variables.
\begin{align*}
\mbox{(PSM) }\left\{\begin{array}{rl}
\ds\min & \ \ \ds\frac{1}{2}\int_{0}^{2\pi} x_1^2(t) + x_2^2(t) + x_3^2(t) + x_4^2(t) + u_1^2(t) + u_2^2(t)\,dt  \\[5mm] 
\mbox{subject to} & \ \ \dot{x}_1(t) = x_2(t)\,,\ \ x_1(0) = 0\,,\ \ x_1(2\pi) = 0\,, \\[2mm]
& \ \ \dot{x}_2(t) = -3x_1(t) + 2x_3(t) + u_1(t)\,,\ \ x_2(0) = 1\,,\ \ x_2(2\pi) = 0\,, \\[2mm]
& \ \ \dot{x}_3(t) = x_4(t)\,,\ \ x_3(0) = 1\,,\ \ x_3(2\pi) = 0\,, \\[2mm]
& \ \ \dot{x}_4(t) = 2x_1(t) - 2x_3(t) + u_2(t)\,,\ \ x_4(0) = -1\,,\ \ x_4(2\pi) = 0\,, \\[2mm]
& \ \ \underline{u}\leq u(t)\leq \overline{u}\,, \ \ \underline{x}\leq x(t)\leq \overline{x}
\end{array} \right.
\end{align*}

\subsection{Numerical discussion and comparisons}
\textbf{Technical Specifications.}
In the numerical experiments we used {\sc Matlab} version 2023b with DR and compared with AMPL--Ipopt optimization computational suite \cite{AMPL,WacBie2006} with Ipopt version 3.12.13. We chose to make comparisons to Ipopt since it is a free and easily obtainable solver used for problems such as those presented in this paper (also see~\cite{BanKay2013}). All computations set $\epsilon=10^{-8}$ from Algorithm \ref{alg:DR} or in the case of Ipopt $\text{tol}=10^{-8}$ and were run on a PC with an i5-10500T 2.30GHz processor with 8GB RAM. For the two examples (PHO) and (PSM) we experimented with a case that only had control constraints and a case that also had an added state constraint. The problem specifications can be found in Table \ref{tbl:probs} along with the choices of $\gamma$ that were used in the implementation of DR.

\begingroup
\setlength{\tabcolsep}{6pt} % Default value: 6pt
\renewcommand{\arraystretch}{1.5} % Default value: 1

\begin{table}[t]
    \centering
    \begin{tabular}{ccll}
        Problem & Case & $\gamma$ & Constraints  \\ \hline \hline
        \multirow{2}{3em}{(PHO)} & Case 1 & $0.60$ & $-0.4\leq u_1(t)\leq0.1, \ -0.5\leq u_2(t)\leq0.1$ \\
                                 & Case 2 & $0.95$ & $-0.4\leq u_1(t)\leq0.1, \ -0.5\leq u_2(t)\leq0.1, \ -0.025\leq x_1(t)$ \\ \hline
        \multirow{2}{3em}{(PSM)} & Case 1 & $0.55$ & $-0.5\leq u_1(t)\leq0.5, \ -0.4\leq u_2(t)\leq0.4$ \\
                                 & Case 2 & $0.95$ & $-0.5\leq u_1(t)\leq0.5, \ -0.4\leq u_2(t)\leq0.4, \ -0.2\leq x_1(t)$ \\ \hline
    \end{tabular}
    \caption{Problem cases.}
    \label{tbl:probs}
\end{table}
\endgroup

\textbf{Generation of ``True'' Solutions.}
Since we do not have analytical solutions to these problems we have generated higher accuracy numerical solutions to our problems that we will use as ``true'' solutions in our error calculations. In the Case 1 examples (only control constraints present) DR was able to successfully converge to an acceptable solution using $N=10^7, \epsilon=10^{-12}$ without reaching 200 iterations (the maximum number of iterations we allowed). The 8GB RAM provided insufficient memory for Ipopt to generate a solution with such a high number of discretization points so DR was used to generate the ``true'' solutions in this case. In Case 2 (state constrained case) DR was unable to converge to an accepted solution in less than 200 iterations so we instead relied on Ipopt. Due to the memory limitations the ``true'' solution for (PHO) Case 2 was generated with $N=10^6,\ \epsilon=10^{-12}$ and (PSM) Case 2 used $N=7\times10^5,\ \epsilon=10^{-12}$.

\textbf{Choice of an Algorithmic Parameter.}
The values of $\gamma$ in Table~\ref{tbl:probs} were decided by generating plots displaying the number of iterations required for DR to find an acceptable solution for 500 values of $\gamma$ in the interval $(0,1)$. For both (PHO) and (PSM) in Case 2 for all values of $\gamma$ that were tested DR required more than the maximum 200 iterations to converge. For these cases we instead generated plots that compared the errors in the states and controls for the different values of $\gamma$. A specific value of $\gamma$ that would provide the smallest errors was not obvious since many values provided similar performance but values closer to 1 appeared optimal thus the choice of $\gamma=0.95$ was made for these experiments.

\begin{figure}[t]
\begin{subfigure}{0.5\textwidth}
    \centering
    \includegraphics[width=7.5cm]{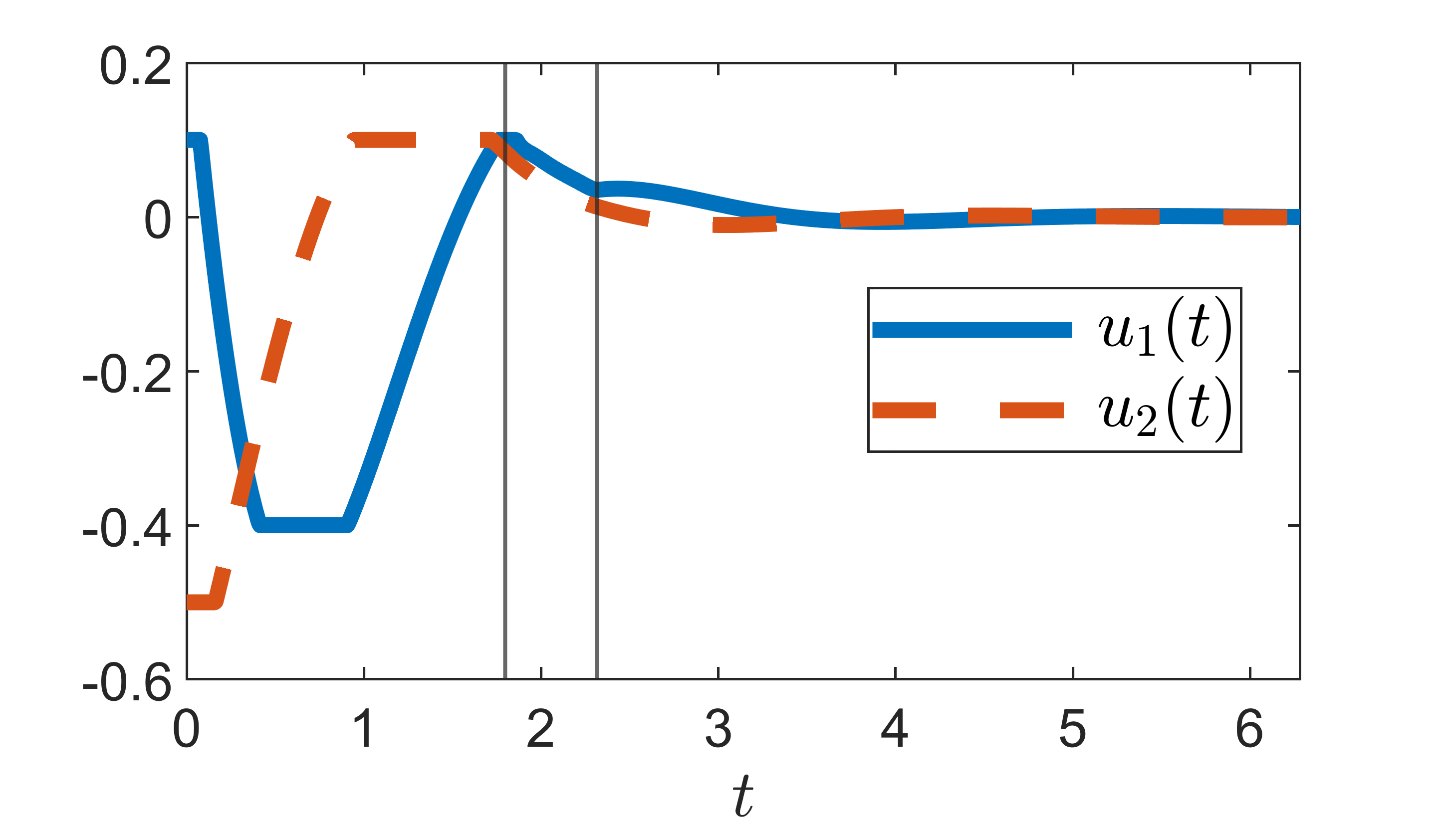}
    \includegraphics[width=7.5cm]{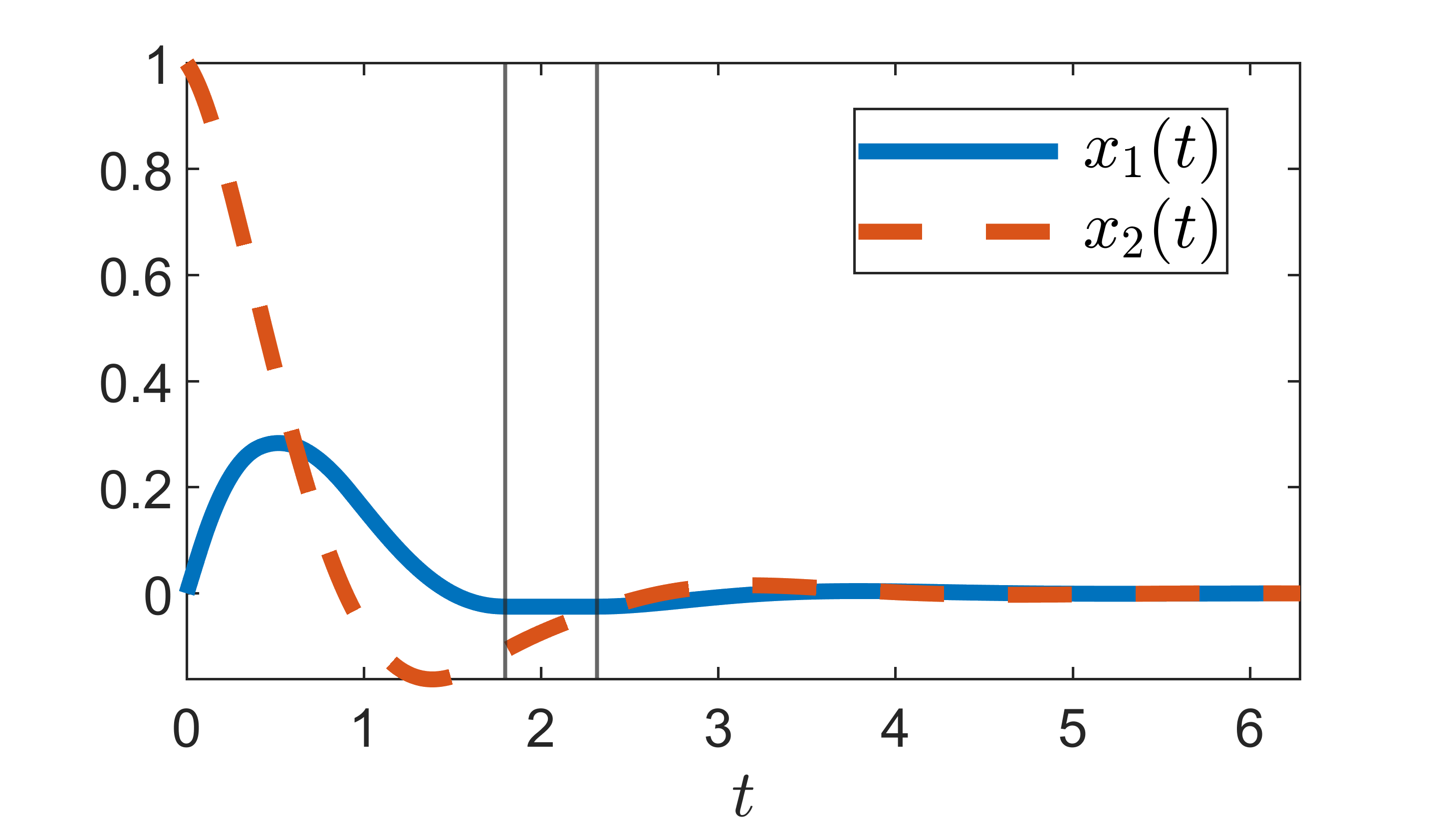}
\end{subfigure}
\begin{subfigure}{0.5\textwidth}
    \centering
    \includegraphics[width=7.5cm]{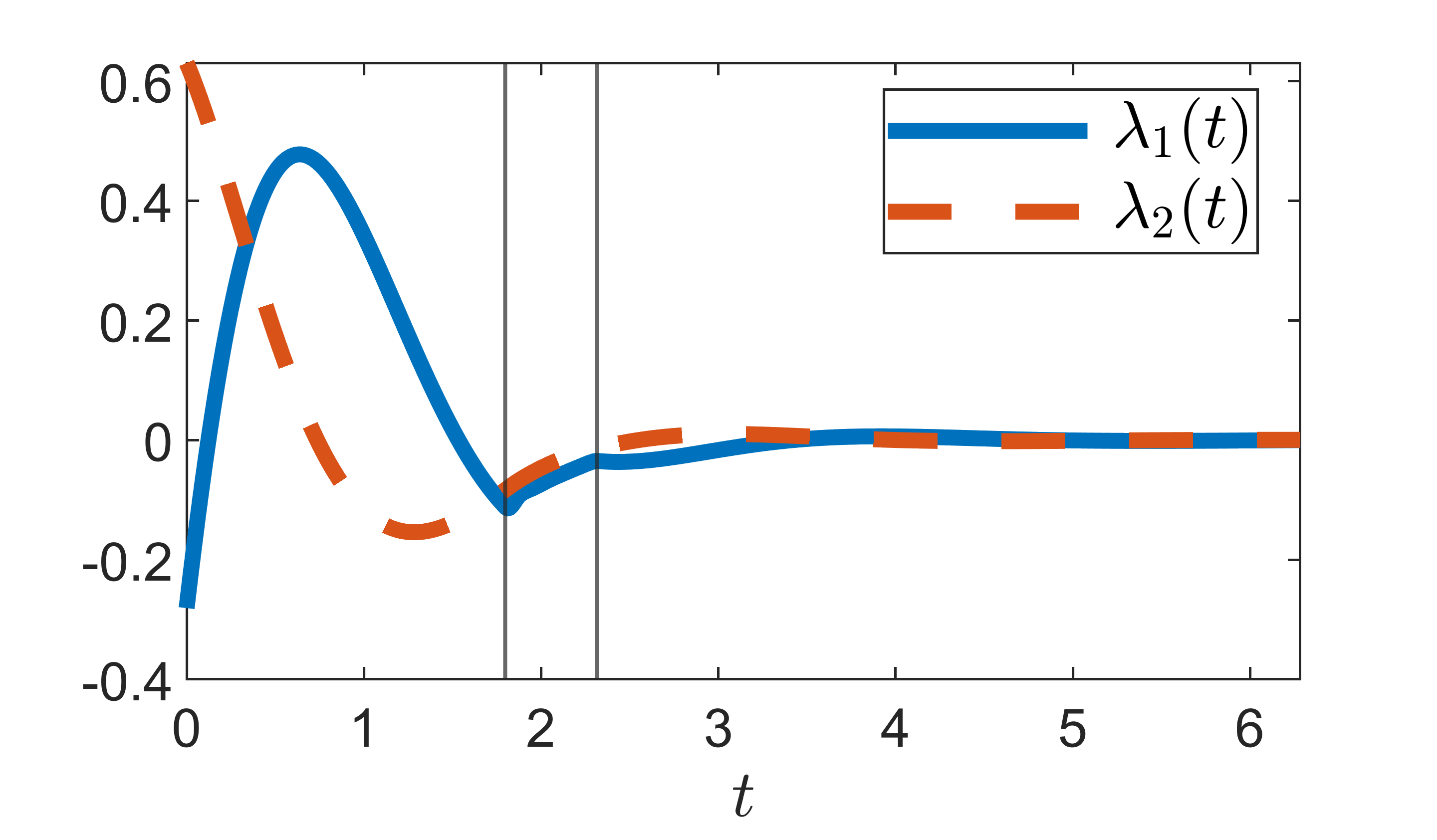}
    \includegraphics[width=7.5cm]{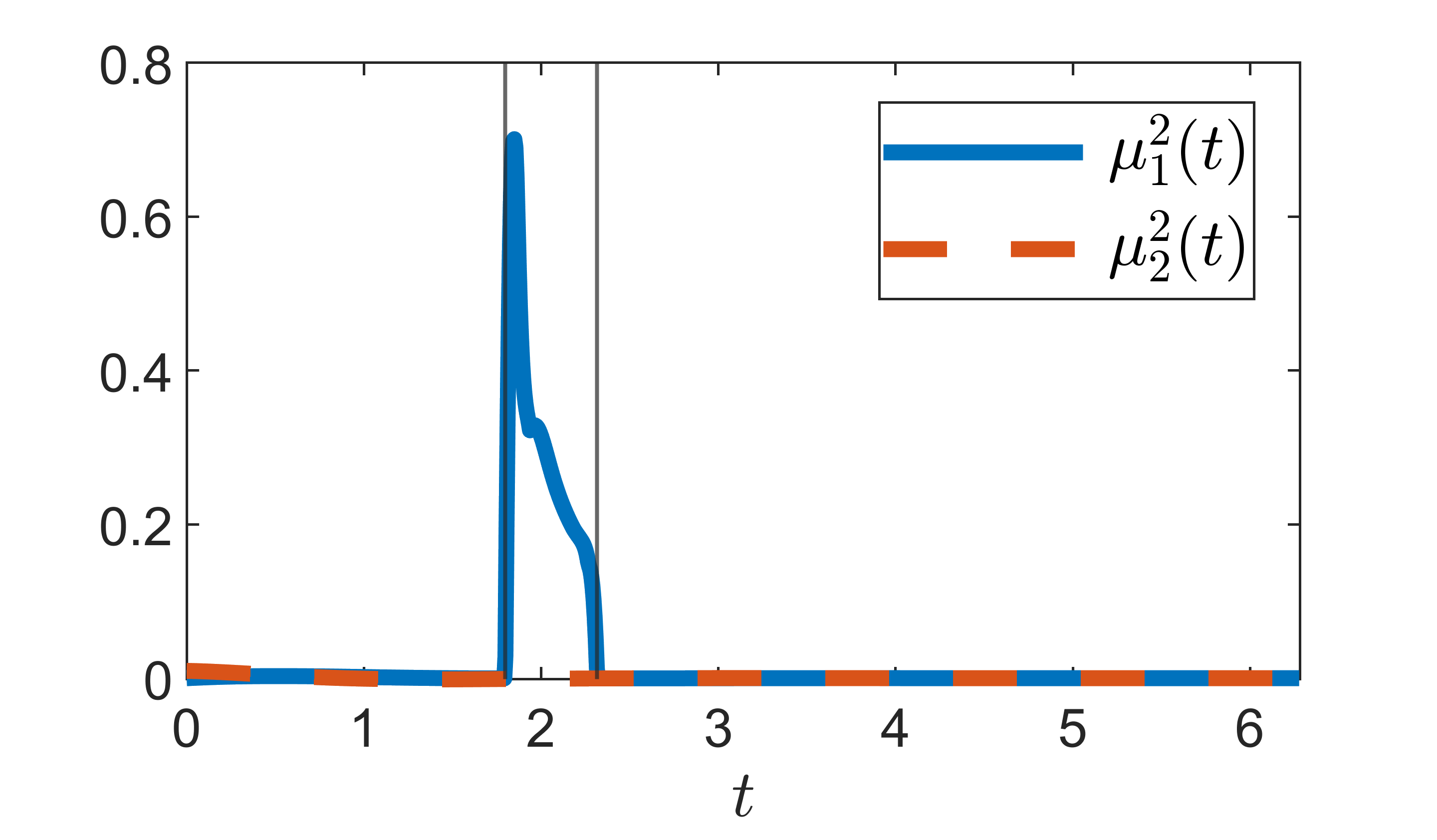}
\end{subfigure}
\caption{(PHO) Case 2 plots (see Table \ref{tbl:probs}) using DR with $N=10^3, -0.025\leq x_1(t)$. Vertical lines indicate the interval in which the state constraint becomes active.}
\label{fig:plots_PHO}
\end{figure}

\begin{figure}[th!]
\begin{subfigure}{0.5\textwidth}
    \centering
    \includegraphics[width=7.5cm]{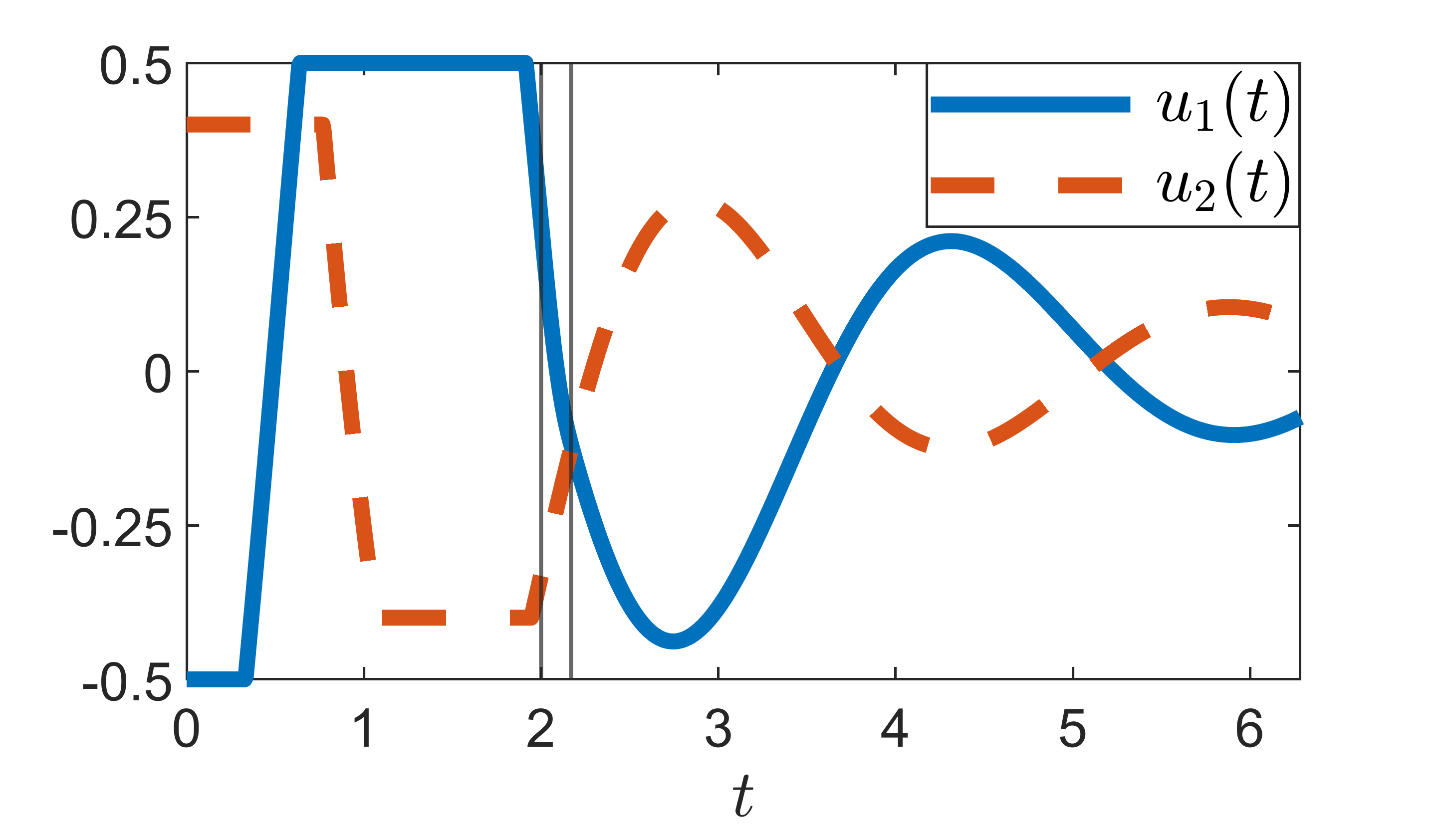}
    \includegraphics[width=7.5cm]{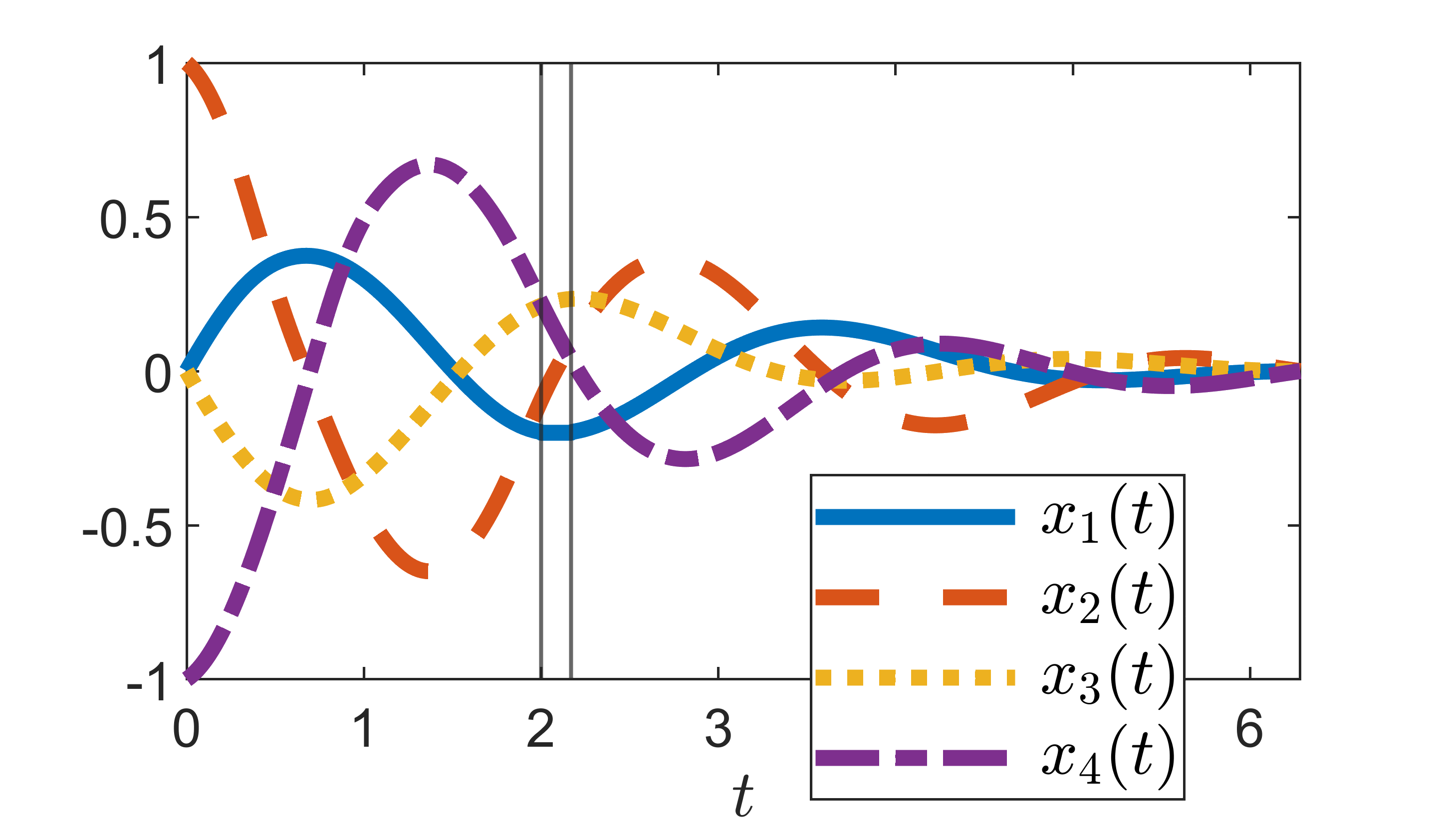}
\end{subfigure}
\begin{subfigure}{0.5\textwidth}
    \centering
    \includegraphics[width=7.5cm]{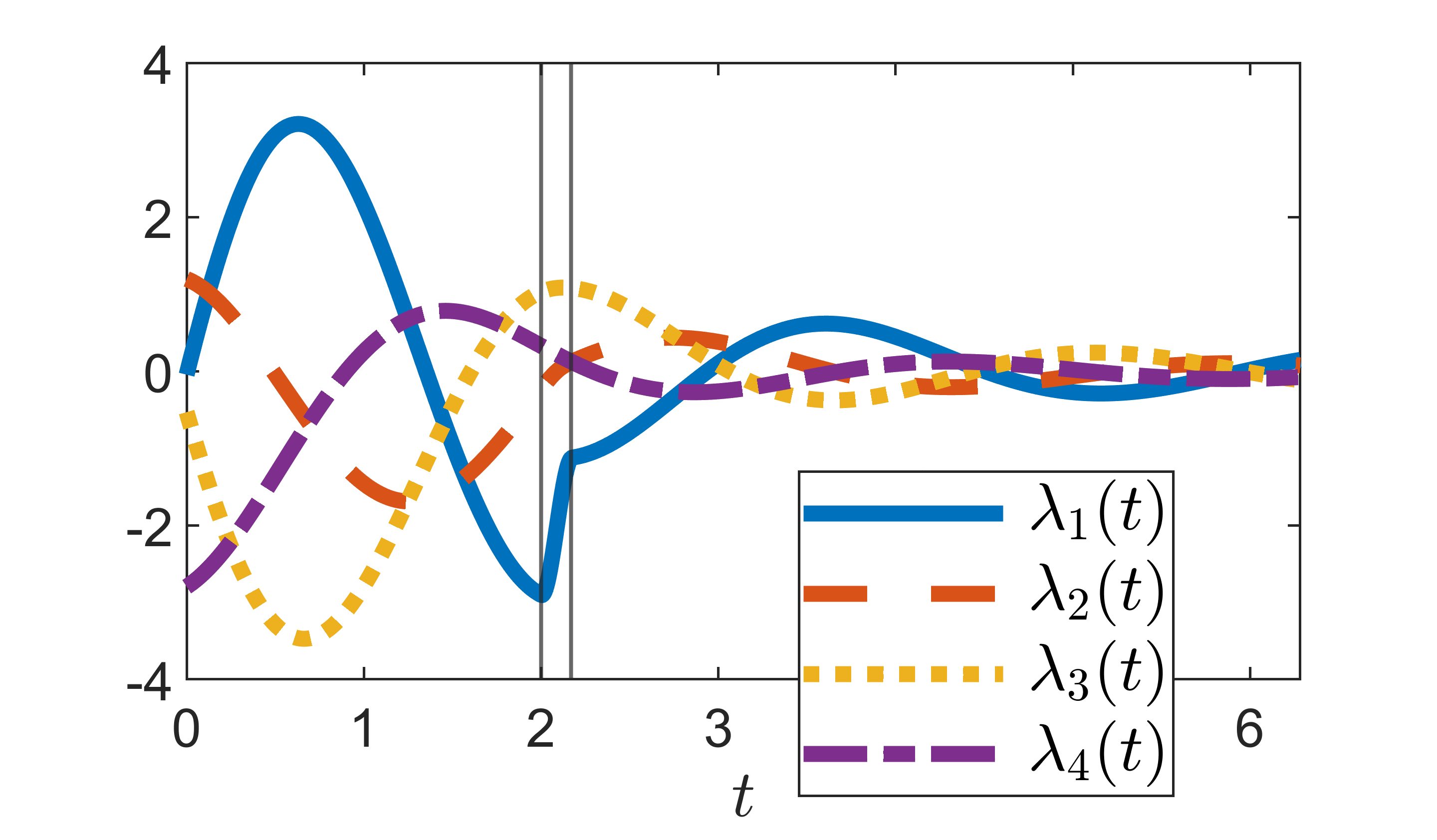}
    \includegraphics[width=7.5cm]{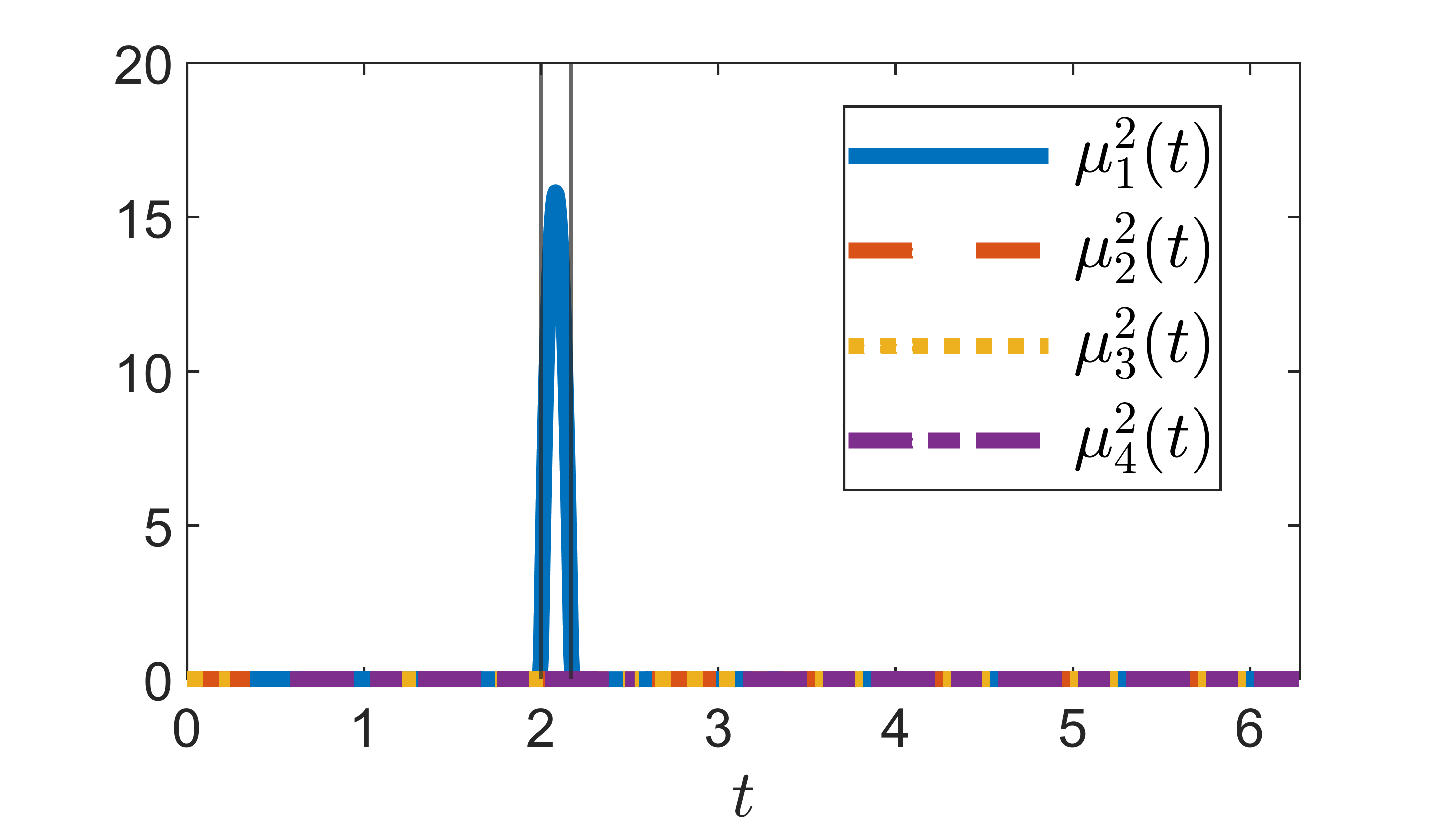}
\end{subfigure}
\caption{(PSM) Case 2 plots (see Table \ref{tbl:probs}) using DR with $N=10^3, -0.2\leq x_1(t)$. Vertical lines indicate the interval in which the state constraint becomes active.}
\label{fig:plots_PSM}
\end{figure}

\begin{figure}[t]
    \centering
    \includegraphics[width=12cm]{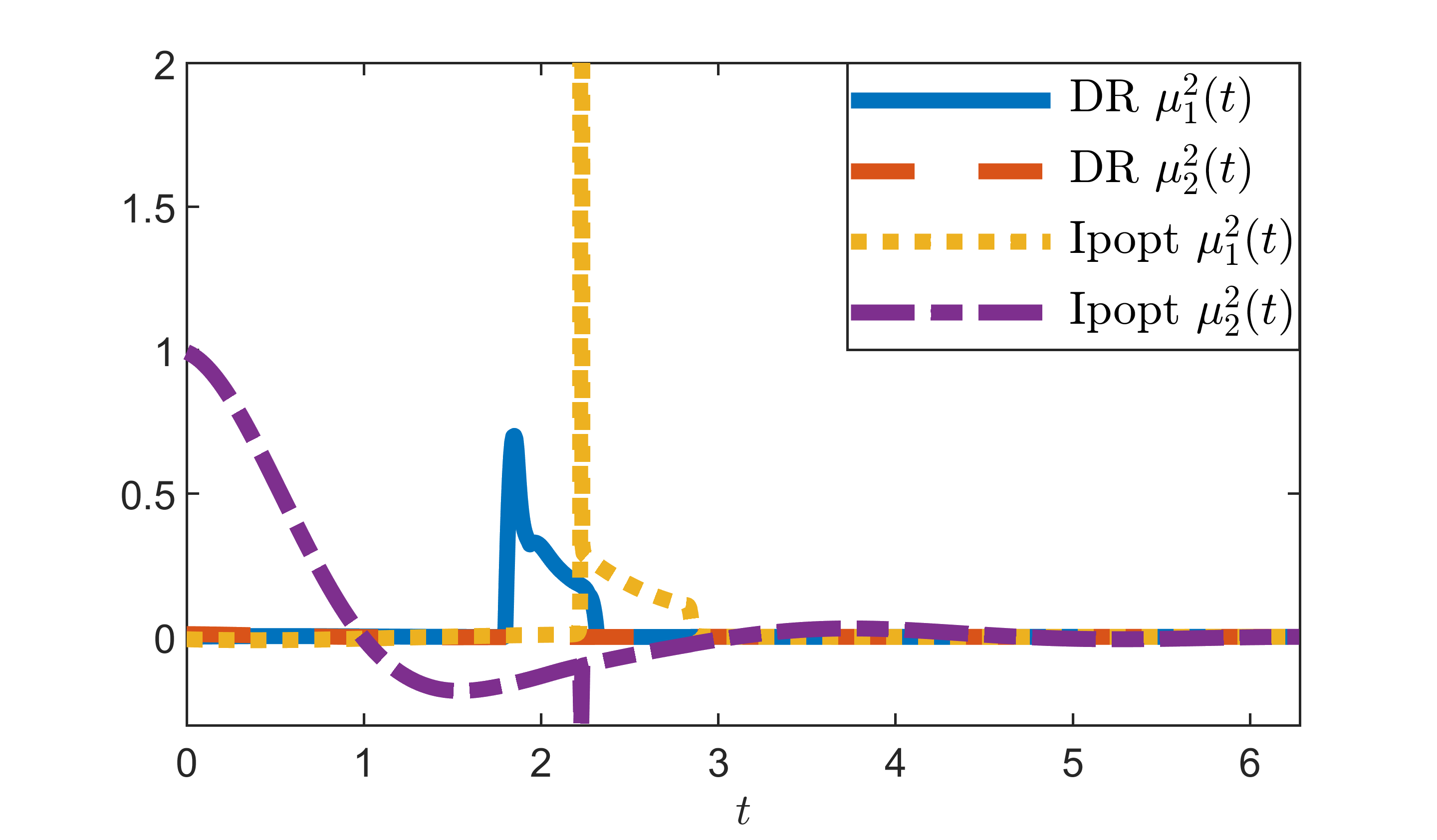}
    \caption{Multipliers $\mu_1^2,\mu_2^2$ for Case 2 (PHO) with $N=10^3$ using DR (solid lines) and Ipopt (dotted lines). Note that $\mu_1^2(t)$ found by Ipopt attains a maximum value of 33, which is not shown in the graph.}
    \label{fig:mu}
\end{figure}

\textbf{Graphical Results.}\ \ 
In Figures~\ref{fig:plots_PHO}--\ref{fig:plots_PSM} we have plots for (PHO) and (PSM) using DR. We have generated similar figures using Ipopt as well but since they mostly overlap Figures~\ref{fig:plots_PHO}--\ref{fig:plots_PSM} we will not show them all here but will point out the differences. In Figure~\ref{fig:mu} we give the multiplier vector $\mu^2$ components for Case 2 (PHO) using DR and Ipopt. As indicated by the black vertical lines in the bottom left plot of Figure~\ref{fig:plots_PHO} when using DR the constraint on $x_1$ is active over a time interval that aligns with the interval where $\mu_1^2$ is positive, as expected from Equations \eqref{eqn:mu1}--\eqref{eqn:mu2}. 

In the results from Ipopt we observe that the state constraint only became active at a single point in time but Figure~\ref{fig:mu} shows that $\mu_1^2$ from Ipopt (yellow dotted line) is positive for a larger interval on time thus Equation \eqref{eqn:mu2} is violated. This appears to be a numerical error that is not present in the DR results since there is an interval of points around the spike where the state constraint $-0.025\leq x_1(t)$ becomes active. When $N=10^4,10^5$ for (PHO) we observe that when using Ipopt the lower bound on $x_1$ is never reached though again we see the interval of points that are almost equal to the lower bound.

In Figure~\ref{fig:mu} we also see a discrepancy in $\mu_2^2$. Since we have not imposed a constraint on the variable $x_2$ Equation~\eqref{eqn:mu2} implies that $\mu_2^2(t)=0$ for all $t\in[t_0,t_f]$. We can see in Figure~\ref{fig:mu} that $\mu_2^2$ from Ipopt (purple dotted line) fails to satisfy this requirement while $\mu_2^2$ from DR is, at least to the eye, equal to zero.

Another difference between the multipliers $\mu$ from DR and Ipopt is the maximum values reached by the functions. In Figure~\ref{fig:mu} we see that $\mu_1^2$ obtained via DR and Ipopt have similar shapes in their plots but the maximum value reached using DR is approximately $0.7$ while from Ipopt the maximum value is approximately $33$. For (PSM) Case 2 with $N=10^3$ the maximum value that $\mu_1^2$ obtained using DR was approximately 16 while Ipopt reached approximately 310. Along with the functions having very different maximum values we noted that when generating these plots for $N=10^3,10^4,10^5$ the results from DR were clearly converging to a single function $\mu_1^2$ while this was not obvious from the Ipopt results. For (PSM) we see the approximate maximum values of $\mu_1^2$ obtained when using Ipopt were 310 for $N=10^3$, 2500 for $N=10^4$ and 3000 for $N=10^5$. In the same example using DR the maximum values of the function were 16 for $N=10^3$, 14 for $N=10^4$ and 14 for $N=10^5$.

We also see some slight variation between the plots from DR and Ipopt at the junction times in the control variables (not shown in this paper in order to avoid excessive amount of visual material). The intervals where the control variables attain their boundaries using DR always appear slightly larger than those from Ipopt. The control variables in the region where they transition between active and inactive constraints appear more rounded at these corners when using Ipopt and have a more sharp transition when using DR.

\textbf{Errors and CPU Times.}
Table~\ref{tbl:exp1} contains the errors in the controls, states, costates for DR and Ipopt while Table~\ref{tbl:exp2} contains the errors in the multipliers, objective values and CPU times. The CPU times were computed as averages of over 200 runs (up to 1,000 runs in the faster examples). The values within boxes are the smaller errors and CPU times between DR and Ipopt. At a glance we can see that more often than not DR produced smaller errors and faster CPU times when compared with Ipopt. Upon closer inspection we see that in many of the Case 2 results the errors from DR and Ipopt are comparable. We note that the ''true'' solutions we are using to calculate these errors are those explained earlier in this subsection except for the multipliers $\mu^2$ in the case 1 examples. In those examples we take the zero vector as our ``true'' solution.

We observe that most of the results in Table~\ref{tbl:exp1} show a smaller error in the control variable from DR, especially in the case 1 examples where there are no state constraints. In the state variables we see that Ipopt has the smaller errors when $N=10^3,10^4$ though there is little difference compared with DR and we see an improvement in DR when $N=10^5$. Like with the error in the control variables we see that the error in the costates is smaller for DR in almost all examples.

From Table~\ref{tbl:exp2} the errors in the multipliers show an improvement in DR compared with Ipopt though the ``true'' solution in the case 2 examples was obtained using results from Ipopt which as previously mentioned did not appear to converge to a specific value as we increased $N$ so the quality of this ``true'' solution is not guaranteed. DR produced slightly smaller objective values that were closest to the ``true'' solution in almost all cases though the difference compared with Ipopt is marginal. We see that the CPU times are faster for DR especially in the examples where $N=10^5$.

\begin{landscape}
\begin{table}[t]
\centering
\small
\begin{tabular}{ScScllllllll}
 & & \multicolumn{2}{|c|}{$L^\infty$ error in control} & \multicolumn{2}{c|}{$L^\infty$ error in states} & \multicolumn{2}{|c}{$L^\infty$ error in costates} \\ \cline{3-4}  \cline{5-6}  \cline{7-8}
    $N$ & Problem & \multicolumn{1}{|c}{DR} & \multicolumn{1}{c|}{Ipopt} & \multicolumn{1}{c}{DR} & \multicolumn{1}{c|}{Ipopt} & \multicolumn{1}{|c}{DR} & \multicolumn{1}{c|}{Ipopt} \\ \hline\hline
    $10^3$ & (PHO) - Case 1 & $\boxed{7.9\times10^{-3}}$ & $1.4\times10^{-2}$ & $6.9\times10^{-3}$ & $\boxed{2.7\times10^{-3}}$ & $\boxed{6.3\times10^{-3}}$ & $1.4\times10^{-2}$ \\
    & (PHO) - Case 2 & $\boxed{1.4\times10^{-2}}$ & $1.6\times10^{-2}$ & $4.5\times10^{-3}$ & $\boxed{2.9\times10^{-3}}$ & $\boxed{1.4\times10^{-2}}$ & $1.6\times10^{-2}$ \\
    & (PSM) - Case 1 & $\boxed{2.3\times10^{-2}}$ & $3.8\times10^{-2}$ & $2.0\times10^{-2}$ & $\boxed{1.8\times10^{-2}}$ & $\boxed{5.5\times10^{-2}}$ & $1.5\times10^{-1}$ \\
    & (PSM) - Case 2 & $\boxed{7.1\times10^{-2}}$ & $1.1\times10^{-1}$ & $3.8\times10^{-1}$ & $\boxed{3.7\times10^{-1}}$ & $\boxed{8.4\times10^{-1}}$ & $1.5\times10^0$ \\
    \hline
    $10^4$ & (PHO) - Case 1 & $\boxed{7.8\times10^{-4}}$ & $3.4\times10^{-3}$ & $6.7\times10^{-4}$ & $\boxed{3.6\times10^{-4}}$ & $\boxed{7.5\times10^{-4}}$ & $1.6\times10^{-3}$ \\
    & (PHO) - Case 2 & $\boxed{1.3\times10^{-3}}$ & $3.0\times10^{-3}$ & $4.7\times10^{-4}$ & $\boxed{4.2\times10^{-4}}$ & $5.0\times10^{-3}$ & $\boxed{4.8\times10^{-3}}$ \\
    & (PSM) - Case 1 & $\boxed{2.2\times10^{-3}}$ & $3.6\times10^{-3}$ & $2.0\times10^{-3}$ & $\boxed{1.8\times10^{-3}}$ & $\boxed{4.8\times10^{-3}}$ & $1.5\times10^{-2}$ \\
    & (PSM) - Case 2 & $2.4\times10^{-2}$ & $\boxed{1.1\times10^{-2}}$ & $\boxed{3.7\times10^{-1}}$ & $\boxed{3.7\times10^{-1}}$ & $\boxed{7.9\times10^{-1}}$ & $1.5\times10^0$ \\
    \hline
    $10^5$ & (PHO) - Case 1 & $\boxed{7.7\times10^{-5}}$ & $6.3\times10^{-3}$ & $\boxed{6.7\times10^{-5}}$ & $9.5\times10^{-4}$ & $\boxed{6.5\times10^{-5}}$ & $2.1\times10^{-3}$ \\
    & (PHO) - Case 2 & $\boxed{7.8\times10^{-4}}$ & $7.5\times10^{-3}$ & $\boxed{6.4\times10^{-5}}$ & $1.4\times10^{-3}$ & $\boxed{5.7\times10^{-3}}$ & $7.3\times10^{-3}$ \\
    & (PSM) - Case 1 & $\boxed{2.2\times10^{-4}}$ & $1.1\times10^{-2}$ & $\boxed{2.0\times10^{-4}}$ & $6.8\times10^{-4}$ & $\boxed{4.3\times10^{-4}}$ & $4.0\times10^{-3}$ \\
    & (PSM) - Case 2 & $2.6\times10^{-2}$ & $\boxed{6.0\times10^{-3}}$ & $\boxed{3.7\times10^{-1}}$ & $\boxed{3.7\times10^{-1}}$ & $8.0\times10^{-1}$ & $\boxed{7.9\times10^{-1}}$ \\
    \hline
    \end{tabular}
    \caption{\sf\small Errors in controls $u$, states $x$ and costates $\lambda$ and CPU times for the DR algorithm and AMPL--Ipopt,\\ with $\varepsilon = 10^{-8}$ and specifications from Table \ref{tbl:probs}.}
    \label{tbl:exp1}
\end{table}
\end{landscape}

\begin{landscape}
\begin{table}[t]
\centering
\small
\begin{tabular}{ScScllllllll}
 & & \multicolumn{2}{|c|}{$L^\infty$ error in $\mu$} & \multicolumn{2}{c|}{$L^\infty$ error in objective values} & \multicolumn{2}{|c}{CPU time [sec]} \\ \cline{3-4}  \cline{5-6}  \cline{7-8}
    $N$ & Problem & \multicolumn{1}{|c}{DR} & \multicolumn{1}{c|}{Ipopt} & \multicolumn{1}{c}{DR} & \multicolumn{1}{c|}{Ipopt} & \multicolumn{1}{|c}{DR} & \multicolumn{1}{c|}{Ipopt} \\ \hline\hline
    $10^3$ & (PHO) - Case 1 & $\boxed{4.9\times10^{-3}}$ & $3.1\times10^{-2}$ & $\boxed{2.9\times10^{-3}}$ & $4.8\times10^{-3}$ & $\boxed{4.8\times10^{-2}}$ & $2.9\times10^{-1}$ \\
    & (PHO) - Case 2 & $\boxed{5.4\times10^{-1}}$ & $1.5\times10^{0}$ & $\boxed{2.9\times10^{-3}}$ & $4.9\times10^{-3}$ & $\boxed{3.3\times10^{-1}}$ & $3.7\times10^{-1}$ \\
    & (PSM) - Case 1 & $\boxed{2.4\times10^{-3}}$ & $8.6\times10^{-2}$ & $\boxed{4.8\times10^{-2}}$ & $5.1\times10^{-2}$ & $\boxed{1.0\times10^{-1}}$ & $4.4\times10^{-1}$ \\
    & (PSM) - Case 2 & $1.6\times10^{1}$ & $\boxed{3.1\times10^{0}}$ & $\boxed{6.8\times10^{-2}}$ & $7.4\times10^{-2}$ & $\boxed{5.2\times10^{-1}}$ & $5.3\times10^{-1}$ \\
    \hline
    $10^4$ & (PHO) - Case 1 & $\boxed{1.5\times10^{-4}}$ & $3.0\times10^{-3}$ & $\boxed{2.8\times10^{-4}}$ & $4.8\times10^{-4}$ & $\boxed{4.7\times10^{-1}}$ & $2.7\times10^{0}$ \\
    & (PHO) - Case 2 & $\boxed{1.7\times10^{1}}$ & $\boxed{1.7\times10^{1}}$ & $\boxed{2.8\times10^{-4}}$ & $4.9\times10^{-4}$ & $3.5\times10^{0}$ & $\boxed{3.3\times10^{0}}$ \\
    & (PSM) - Case 1 & $\boxed{7.1\times10^{-4}}$ & $8.2\times10^{-3}$ & $\boxed{4.6\times10^{-3}}$ & $5.0\times10^{-3}$ & $\boxed{1.1\times10^{0}}$ & $4.6\times10^{0}$ \\
    & (PSM) - Case 2 & $1.4\times10^{1}$ & $\boxed{2.5\times10^{0}}$ & $\boxed{4.4\times10^{-3}}$ & $6.9\times10^{-3}$ & $\boxed{5.9\times10^{0}}$ & $6.4\times10^{0}$ \\
    \hline
    $10^5$ & (PHO) - Case 1 & $\boxed{3.7\times10^{-5}}$ & $3.0\times10^{-4}$ & $\boxed{2.8\times10^{-5}}$ & $7.7\times10^{-5}$ & $\boxed{5.1\times10^{0}}$ & $2.6\times10^{1}$ \\
    & (PHO) - Case 2 & $\boxed{1.8\times10^{1}}$ & $\boxed{1.8\times10^{1}}$ & $\boxed{2.4\times10^{-5}}$ & $1.4\times10^{-4}$ & $3.7\times10^{1}$ & $\boxed{3.1\times10^{1}}$ \\
    & (PSM) - Case 1 & $\boxed{1.1\times10^{-4}}$ & $8.2\times10^{-4}$ & $\boxed{4.5\times10^{-4}}$ & $5.3\times10^{-4}$ & $\boxed{1.2\times10^{1}}$ & $1.2\times10^{2}$ \\
    & (PSM) - Case 2 & $\boxed{2.2\times10^{2}}$ & $3.7\times10^{3}$ & $1.5\times10^{-3}$ & $\boxed{7.2\times10^{-4}}$ & $\boxed{6.6\times10^{1}}$ & $1.6\times10^{2}$ \\
    \hline
    \end{tabular}
    \caption{\sf\small Errors in multipliers $\mu$, errors in objective values and CPU times for the DR algorithm and AMPL--Ipopt,\\ with $\varepsilon = 10^{-8}$ and specifications from Table \ref{tbl:probs}.}
    \label{tbl:exp2}
\end{table}
\end{landscape}

\section{Conclusion}
\label{sec:concl}

We have applied the Douglas--Rachford (DR) algorithm to find a numerical solution of LQ control problems with state and control constraints, after re-formulating these problems as the sum of two convex functions and deriving expressions for the proximal mappings of these functions~(Theorems~\ref{thm:proxf}--\ref{thm:proxg}).  These proximal mappings were used in the DR iterations.  Within the DR algorithm~(Algorithm~\ref{alg:DR}), we proposed a procedure~(Algorithm~\ref{alg:projA}) for finding the projection onto the affine set defined by the ODEs numerically.  

We carried out extensive numerical experiments on two nontrivial example problems and illustrated both the working of the algorithm and its efficiency (in both accuracy and speed) compared with the traditional approach of direct discretization. We observed that in general the DR algorithm produced
smaller errors and faster run times for these problems most notably in the examples where we
have increased the number of discretization points. From these numerical results the DR algorithm could in general
be recommended over Ipopt when generating high quality solutions.

Based on further extensive experiments, we conjectured on how the costate variables can be determined.  We successfully used the costate variables constructed as in the conjecture, as well as the state constraint multipliers calculated using these costate variables, for the numerical verification of the optimality conditions.

We recall that Algorithm~\ref{alg:projA} involves repeated numerical solution of the ODEs in~\eqref{eqn:lin_sys} with various initial conditions.  For solving~\eqref{eqn:lin_sys} we implemented (explicit) Euler's method which requires only a continuous right-hand side of the ODEs.  Algorithm~\ref{alg:projA} appears to be successful for the worked examples partly owing to the fact that in these examples the optimal control is continuous.  We tried to use our approach for the machine tool manipulator problem from \cite{MTM} which has 7 state variables, one control variable, and upper and lower bounds imposed on the single control variable and one of the state variables.  However, our approach did not seem to yield a solution (so far) for this problem, conceivably owing to the fact that the optimal control variable, as well as the optimal costate variable vector, is not continuous, in that these variables have jumps a number of times during the process---see~\cite[Figure~5]{MTM}.  Note that discontinuities in the control make the right-hand side of~\eqref{eqn:lin_sys} discontinuous rendering Euler's method ineffective.  Therefore, problems of the kind in~\cite{MTM} require further investigation.

We believe that our approach can be extended to more general convex optimal control problems, for example those with a non-quadratic objective function or mixed state and control constraints, as long as the pertaining proximal operators are not too challenging to derive.  
%Having said this, in the current paper, we derived proximal mappings for the case when $Q$ and $R$ in the quadratic objective functional are diagonal.  Even for the LQ control problems with non-diagonal $Q$ and $R$, the derivation of the proximal mappings seems to be quite more challenging.

It would also be interesting to employ and test, in the future, other projection type methods, for example the Peaceman--Rachford algorithm~\cite[Section~26.4 and Proposition~28.8]{BauCom2017}, which, to the knowledge of the authors, has not been applied to optimal control problems.

\section*{Acknowledgments}

BIC was supported by an Australian Government Research Training Program Scholarship. No funding
was received by RSB and CYK to assist with the preparation of this manuscript.

\section*{Conflict of Interest}

The authors have no competing, or conflict of, interests to declare that are relevant to the content of this article.

\section*{Use of AI Tools Declaration}

The authors declare that they have not used Artificial Intelligence (AI) tools in the creation of this article.


{\small
\begin{thebibliography}{30}

\bibitem{AmmKen1998}
{\sc H. M. Amman, D. A. Kendrick},
{\em Computing the steady state of linear quadratic optimization models with rational expectations}.
Econ. Lett., 58(2), 185--191, 1998.

\bibitem{AraBorTam2014}
{\sc F. J. Aragón Artacho, J. M. Borwein and M. K. Tam},
{\em Douglas--Rachford feasibility methods For matrix completion problems}.
ANZIAM Journal, 55(4), 299--326, 2014.

\bibitem{AraCampEls2020}
{\sc F. J. Aragón Artacho, R. Campoy and V. Elser},
{\em An enhanced formulation for solving graph coloring problems with the Douglas--Rachford algorithm}.
J. Glob. Optim., 77(2), 383--403, 2020.

\bibitem{Ascher95}
{\sc U.~M. Ascher, R.~M.~M. Mattheij, and  R.~D. Russell},
{\em Numerical Solution of Boundary Value Problems for Ordinary Differential Equations}. SIAM Publications, Philadelphia, 1995.

\bibitem{BanKay2013}
{\sc N.~Banihashemi and C.~Y.~Kaya},
{\em Inexact restoration for Euler discretization of box-constrained optimal control problems}.
J. Optim. Theory Appl., 156 (2013), 726--760.

\bibitem{Bau2008}
{\sc H. H. Bauschke},
{\em 8 Queens, Sudoku, and Projection Methods}.
\url{https://carma.newcastle.edu.au/resources/jon/Preprints/Books/CUP/Material/Lions-Mercier/Heinz_Bauschke.pdf}, 2008.

\bibitem{BausBuraKaya2019}
{\sc H.~H.~Bauschke, R.~S.~Burachik, and C.~Y.~Kaya},
{\em Constraint splitting and projection methods for optimal control of double integrator}
in { Splitting Algorithms, Modern Operator Theory, and Applications}.
Springer, 45--68, 2019.

\bibitem{BauCom2017}
{\sc H.~H.~Bauschke and P.~L.~Combettes},
{\em Convex Analysis and Monotone Operator Theory in Hilbert
Spaces}, Second edition. Springer, 2017. 

\bibitem{BauKoch}
{\sc H. H. Bauschke and V. R. Koch},
{\em Projection methods: Swiss Army knives for solving
feasibility and best approximation problems with halfspaces}.
Contemporary Mathematics, 636,
1--40, 2015. 

\bibitem{BauMoursi}
{\sc H.~H.~Bauschke and W.~M.~Moursi},
{\em On the Douglas--Rachford algorithm}.
Math. Program., Ser. A, 164, 263--284, 2017.

\bibitem{BurCalKay2022}
{\sc R.~S.~Burachik, B.~I.~Caldwell, C.~Y.~Kaya},
{\em Projection methods for control-constrained
minimum-energy control problems}. 
arXiv:2210.17279v1. https://arxiv.org/abs/2210.17279, 2022.

\bibitem{BurCalKay2024}
{\sc R.~S.~Burachik, B.~I.~Caldwell, C.~Y.~Kaya},
{\em Douglas--Rachford algorithm for control-constrained minimum-energy control problems}.  To appear in ESAIM Control Optim. Calc. Var.
arXiv:2210.17279v2, https://arxiv.org/abs/....., 2024.

\bibitem{BurCalKayMou2023}
{\sc R.~S.~Burachik, B.~I.~Caldwell, C.~Y.~Kaya, W.~M.~Moursi},
{\em Optimal control duality and the Douglas--Rachford algorithm}. 
To appear in SIAM J. Control Optim.  \\
https://doi.org/10.48550/arXiv.2303.06527, 2023.

\bibitem{BurKayMaj2014}
{\sc R. S. Burachik, C. Y. Kaya, and S. N. Majeed},
{\em A duality approach for solving control-constrained linear-quadratic optimal control problems}.
SIAM J. Control Optim., 52 (2014), 1771--1782.

\bibitem{BusMau2000}
{\sc C. B\"{u}skens, H. Maurer},
{\em SQP-methods for solving optimal control problems with control and state constraints: Adjoint variables, sensitivity analysis and real-time control}.
J. Comput. Appl. Math., 120(1), 85--108, 2000.

\bibitem{MTM}
{\sc B. Christiansen, H. Maurer, and O. Zirn},
{\em Optimal control of machine tool manipulators}.  Diehl M, Glineur F, Jarlebring E, Michiels W (eds), Recent Advances in Optimization and Its Applications in Engineering. Springer, Berlin, Heidelberg, 451–460, 2010.

\bibitem{DougRach}
{\sc J.~Douglas and H.~H.~Rachford},
{\em On the numerical solution of heat conduction problems in two and three space variables}.
Trans. Amer. Math. Soc., 82, 421--439, 1956. 

\bibitem{EckBer}
{\sc J.~Eckstein and D.~P.~Bertsekas},
{\em On the Douglas-Rachford splitting method and the proximal
point algorithm for maximal monotone operators}.
Math. Progr., Ser. A, 55, 293--318, 1992. 

\bibitem{AMPL}
{\sc R.~Fourer, D.~M.~Gay, and B.~W.~Kernighan},
{\em AMPL: A Modeling Language for Mathematical Programming}, 2nd edition.
Brooks/Cole Publishing Company / Cengage Learning, 2003.

\bibitem{GraEls2008}
{\sc S. Gravel and V. Elser},
{\em Divide and concur: A general approach to constraint satisfaction}.
Physical Review E, Statistical, Nonlinear, and Soft Matter Physics, 78(3), 036706--036706, 2008.

\bibitem{HarSetVic1995}
{\sc R. F. Hartl, S. P. Sethi, R. G. Vickson},
{\em A survey of the maximum principles for optimal control problems with state constraints}.
SIAM Review, 52(2), 181--218, 1995.

\bibitem{Kaya2010}
{\sc C.~Y.~Kaya},
{\em Inexact Restoration for Runge-Kutta discretization of optimal control problems},
SIAM J. Numer. Anal., 48(4), 1492--1517, 2010.

\bibitem{Keller68}
{\sc H.~B.~Keller},
{\em Numerical Methods for Two-Point Boundary-Value Problems}.
Blaisdell, London, 1968.

\bibitem{KugPes1990}
{\sc B. Kugelmann, H. J. Pesch},
{\em New general guidance method in constrained optimal control. I, Numerical method}.
J. Optim. Theory Appl., 67(3), 421--435, 1990.

\bibitem{LM}
{\sc P.-L. Lions and B.~Mercier},
{\em Splitting algorithms for the sum of two nonlinear
operators}.
SIAM J. Numer. Anal., 16, 964--979, 1979. 

\bibitem{MauKimVos2005}
{\sc H.~Maurer, J.-H.R.~Kim and G.~Vossen},
{\em On a state-constrained control problem in optimal production and maintenance}.
Optimal Control and Dynamic Games, Applications in Finance, Management Science and Economics, (C.
Deissenberg, R.F. Hartl, eds.), Springer Verlag, 289--308, 2005.

\bibitem{MauObe2003}
{\sc H. Maurer, H. J. Oberle},
{\em Second order sufficient conditions for optimal control problems with free final time: The Riccati approach}.
SIAM J. Control Optim., 41 (2), 380--403, 2003.

\bibitem{Mou2011}
{\sc T. Mouktonglang},
{\em Innate immune response via perturbed LQ-control problem}.
Advanced Studies in Biology, 3, 327--332, 2011.

\bibitem{Rugh1996}
{\sc W.~J.~Rugh},
{\em Linear System Theory}, 2nd edition, Pearson, 1996.

\bibitem{SchSco2021}
{\sc T.~L.~Schmitz and K.~S.~Smith},
{\em Mechanical Vibrations Modeling and Measurement}, 2nd edition. Springer, 2021.

\bibitem{Stoer93}
{\sc J.~Stoer and R.~Bulirsch},  
{\em Introduction to Numerical Analysis}, 2nd edition.
Springer-Verlag, New York, 1993.

\bibitem{Svaiter}
{\sc B. F. Svaiter},
{\em On weak convergence of the Douglas-Rachford method}.
SIAM J. Control Optim., 49, 280--287, 2011. 

\bibitem{WacBie2006}
{\sc A.~W\"achter and L.~T.~Biegler},
{\em On the implementation of a primal-dual interior point filter line
search algorithm for large-scale nonlinear programming}.
Math. Progr., 106, 25--57, 2006.

\end{thebibliography}
}
\end{document}